\documentclass{amsart}
\usepackage{amsmath,amssymb,amsthm}
\usepackage{enumerate}
\newtheorem{thm}{Theorem}[section]
\newtheorem*{main}{Main Theorem 1}
\newtheorem*{main2}{Main Theorem 2}
\newtheorem{prop}[thm]{Proposition}
\newtheorem{lem}[thm]{Lemma}
\newtheorem{cor}[thm]{Corollary}
\theoremstyle{definition}
\newtheorem{df}[thm]{Definition}

\theoremstyle{remark}
\newtheorem{rem}[thm]{Remark}

\newcommand{\dG}{\widehat G}
\newcommand{\cO}{\mathcal O}
\newcommand{\drho}{\widehat \rho}
\newcommand{\dlambda}{\widehat \lambda}
\newcommand{\End}{\operatorname{\rm End}}
\newcommand{\Hom}{\operatorname{\rm Hom}}
\newcommand{\Ann}{\operatorname{\rm Ann}}
\newcommand{\fg}{\mathfrak g}
\newcommand{\fh}{\mathfrak h}
\newcommand{\fn}{\mathfrak n}
\newcommand{\fb}{\mathfrak b}
\newcommand{\daZ}{\pi i \log(q_\alpha)^{-1} {\mathbb Z}}
\newcommand{\dD}{\widehat \Delta}
\newcommand{\de}{\widehat \varepsilon}
\newcommand{\dS}{\widehat S}
\newcommand{\dphi}{\widehat \varphi}
\newcommand{\frsl}{\mathfrak{sl}}
\newcommand{\cU}{\mathcal U}
\newcommand{\cP}{\mathcal P}
\newcommand{\bH}{\mathbb H}
\newcommand{\ad}{\mathrm{ad}}

\renewcommand{\Re}{\mathop{\rm Re}}
\newcommand{\CGq}{\mathrm{Char}(G_q)}
\newcommand{\CqGq}{\mathrm{Char}_q(G_q)}
\begin{document}
	\title[Unitary spherical representations of Drinfeld doubles]{Unitary spherical representations \\
	of Drinfeld doubles}
	\author{Yuki Arano}
	\address{Department of Mathematical Sciences, the University of Tokyo, Komaba, Tokyo, 153-8914, Japan}
	\email{arano@ms.u-tokyo.ac.jp}
	\date{}
	\begin{abstract}
	We study irreducible spherical unitary representations of the Drinfeld double of the $q$-deformation of a connected simply connected compact Lie group,
	which can be considered as a quantum analogue of the complexification of the Lie group.
	In the case of $SU_q(3)$, we give a complete classification of such representations.
	As an application, we show the Drinfeld double of the quantum group $SU_q(2n+1)$ has property (T), which also implies central property (T) of the dual of $SU_q(2n+1)$.
	\end{abstract}
	\maketitle
	\section{Introduction}
	As has been observed by many authors (see, e.g., \cite{Kra}, \cite{PW}), the Drinfeld double of the $q$-deformation of a compact Lie group can be regarded as a quantization of the complexification of the original Lie group. Using this point of view, in this paper we study irreducible spherical unitary representations of these Drinfeld doubles and obtain a complete classification of such representations in the case of $SU_q(3)$.
	Our first main result is as follows. (cf.\ \cite{Duf} for the classical case.)
	\begin{main}
	Fix $0 < q < 1$ and let
	\[\fh^* := \{(\nu_1,\nu_2,\nu_3) \in {\mathbb C}^3 \mid \nu_1 + \nu_2 + \nu_3 = 0\},\]
	\[Q := \fh^* \cap {\mathbb Z}^3, P := \{\nu \in \fh^* \mid (\nu,\mu) \in {\mathbb Z} \text{ for } \mu \in Q\},\]
	\[X := \fh^* / 2 \pi i \log(q)^{-1} Q.\]
	Consider the natural permutation group $W = S_3$-action on $X$.
	Then the irreducible unitary spherical representations of the Drinfeld double of $SU_q(3)$ are parametrized by $\nu \in X/W$ such that
	\begin{enumerate}[{\rm (i)}]
	\item $\nu$ is imaginary (unitary principal series),
	\item $\nu = (t+is,-t+is,-2is)$ modulo $2\pi i \log(q)^{-1} P \times W$ for $t \in [-1,1]$ and $s \in {\mathbb R}$ (complementary series) or
	\item $\nu = (2,-2,0)$ modulo $2\pi i \log(q)^{-1} P \times W$ (characters including the trivial).
	\end{enumerate}
	\end{main}
	We note that the classification of irreducible unitary representations of the Drinfeld double of the  $q$-deformation is initiated by Pusz \cite{Pus} in the case of $SU_q(2)$. 
	In \cite{VY}, Voigt and Yuncken independently obtained similar constructions of irreducible unitary principal series also in nonspherical cases.

	As another consequence, we show the isolation of the trivial representation in the case of $SU_q(2n+1)$.
	\begin{main2}
	Fix $0 < q < 1$. For $n=1,2,\dots$, the Drinfeld double of $SU_q(2n+1)$ has property ${\rm (T)}$.
	\end{main2}

	Let us explain our motivation from operator algebras.
	The notion of compact quantum groups was introduced by Woronowicz \cite{Wor} and has had a great influence in the theory of operator algebras.
	Among properties of compact (equivalently, discrete) quantum groups, approximation properties have a strong link with the structures of the corresponding operator algebras, for example, \cite{Bra}, \cite{Fre} and \cite{Iso}.

	In \cite{DFY}, De Commer, Freslon and Yamashita pointed out that the approximation properties for the Drinfeld doubles of compact quantum groups appeared to be closely related to that of the dual discrete quantum group.
	In the same paper, they showed the central Haagerup property and the central CCAP for the dual of $SU_q(2)$.
	A main application of the result in this paper is to show central property (T), which is contrary to the central Haagerup property, for the dual of $SU_q(2n+1)$, although it is amenable and has commutative fusion rules.

	{\bf Acknowledgement} The author wishes to express his gratitude to Kenny De Commer, Hironori Oya and Makoto Yamashita for many fruitful discussions.
	He is grateful to Kenny De Commer, Yasuyuki Kawahigashi, Christian Voigt, Makoto Yamashita and Robert Yuncken for many valuable comments and pointing out mistakes and typos in the draft version of this paper. He also appreciates the supervision of Yasuyuki Kawahigashi.
	He is also greatly indebted to the anonymous referee and Joachim Cuntz for correcting a large number of errors in an earlier version of the manuscript.
	This work was supported by the Research Fellow of the Japan
Society for the Promotion of Science and the Program for Leading Graduate Schools, MEXT, Japan
	\section{Preliminaries}
	\subsection{Quantized enveloping algebras}
	Throughout this paper, the field is $\mathbb C$. By ${\mathbb Z}_+$, we mean $\{1,2,\dots\}$. We fix $0 < q  < 1$.

	Let $G$ be a connected simply connected compact Lie group, $\fg$ the complexification of its Lie algebra, $\fh$ a Cartan subalgebra of $\fg$, $\Delta$ the set of roots, $Q \subset \fh^*$ the root lattice and $P \subset \fh^*$ the weight lattice.
	Let $(\cdot,\cdot)$ be the natural bilinear form on $\fh$, which is normalized as $(\alpha,\alpha)=2$ for a short root $\alpha$. For each $\alpha \in \Delta$, let $\alpha^\vee := 2\alpha/(\alpha,\alpha)$ be the coroot. We fix a set $\Pi$ of simple roots and let $\Delta_+(Q_+, P_+)$ be the set of positive roots (positive elements in the root lattice, positive weights) with respect to $\Pi$.

	Put $q_\alpha := q^{(\alpha,\alpha)/2}$,
	\[n_q := \frac{q^n - q^{-n}}{q - q^{-1}},\]
	\[n_q ! := n_q (n-1)_q \dots 1_q,\]
	\[\left( \begin{matrix} n \\ m \end{matrix} \right)_q := \frac{n_q!}{m_q! (n-m)_q!}.\]
	\begin{df}
	The quantized enveloping algebra $U_q(\fg)$ is the algebra defined by generators $\{K_\lambda, E_\alpha, F_\alpha \mid \lambda \in P, \alpha \in \Pi\}$ and relations
	\[K_0 = 1, \quad K_\lambda K_\mu = K_{\lambda + \mu},\]
	\[K_\lambda E_\alpha K_{-\lambda} = q^{(\alpha,\lambda)} E_\alpha,\quad K_\lambda F_\alpha K_{-\lambda} = q^{-(\alpha,\lambda)} F_\alpha,\]
	\[[E_\alpha,F_\beta] = \delta_{\alpha,\beta} \frac{K_\alpha - K_{-\alpha}}{q_\alpha - q_\alpha^{-1}},\]
	\[\sum_{r=0}^{1-(\alpha,\beta^\vee)} (-1)^r \left( \begin{matrix} 1-(\alpha,\beta^\vee) \\ r \end{matrix} \right)_{q_\beta} E_\beta^r E_\alpha E_\beta^{1-(\alpha,\beta^\vee)-r} = 0,\]
	\[\sum_{r=0}^{1-(\alpha,\beta^\vee)} (-1)^r \left( \begin{matrix} 1-(\alpha,\beta^\vee) \\ r \end{matrix} \right)_{q_\beta} F_\beta^r F_\alpha F_\beta^{1-(\alpha,\beta^\vee)-r} = 0.\]
	One may define a Hopf $*$-algebra structure on $U_q(\fg)$ by
	\[\dD(K_\lambda) = K_\lambda \otimes K_\lambda, \quad \de(K_\lambda) = 1, \quad \dS(K_\lambda) = K_{-\lambda},\]
	\[\dD(E_\alpha) = E_\alpha \otimes 1 + K_\alpha \otimes E_\alpha,\quad \de(E_\alpha) = 0,\quad \dS(E_\alpha) = -K_{-\alpha} E_\alpha,\]
	\[\dD(F_\alpha) = F_\alpha \otimes K_{-\alpha} + 1 \otimes F_\alpha,\quad \de(F_\alpha)=0,\quad \dS(F_\alpha) = -F_\alpha K_\alpha,\]
	and the $*$-structure is given by
	\[K_\lambda^* = K_\lambda,\quad E_\alpha^* = F_\alpha K_\alpha,\quad F_\alpha^* = K_{-\alpha}E_\alpha.\]
	\end{df}
	Let $U_q(\fn^+)$(resp.\ $U_q(\fh)$, $U_q(\fn^-)$) be the subalgebra of $U_q(\fg)$ generated by $E_\alpha$'s (resp.\ $K_\lambda$'s, $F_\alpha$'s). Then the multiplication
	\[U_q(\fn^-) \otimes U_q(\fh) \otimes U_q(\fn^+) \to U_q(\fg)\]
	is an isomorphism as vector spaces.
	Through this isomorphism, $\cP := \de|_{U_q(\fn^-)} \otimes {\rm id} \otimes \de|_{U_q(\fn^+)}$ is a  projection onto $U_q(\fh)$.

	For each $\lambda \in \fh^*$, let $V(\lambda)$ be the unique irreducible highest module of highest weight $\lambda$, that is, there exists nonzero $v_\lambda \in V(\lambda)$ such that
	\[K_\mu v_\lambda = q^{(\mu,\lambda)} v_\lambda, E_\alpha v_\lambda = 0.\]
	We also denote the corresponding representation by $\pi^\lambda$.
	If $\lambda \in P_+$, then $V(\lambda)$ is finite dimensional. We say a $U_q(\fg)$-module is of {\it type 1} if it decomposes into a direct sum of $V(\lambda)$'s for $\lambda \in P_+$. Notice that any subquotient of type 1 module is also of type 1.

	On the other hand, if $\mu \in \pi i \log(q)^{-1} P$, $V(\mu)$ is $1$-dimensional. Hence for $\lambda \in \fh^*$ and $\mu \in \pi i \log(q)^{-1} P$,
	\[V(\lambda + \mu) = V(\lambda) \otimes V(\mu).\]
	In particular, for $\lambda \in P_+ + \pi i \log (q)^{-1} P$, $V(\lambda)$ is finite dimensional. 
	Notice that
	\[V(\lambda + \mu)^* \otimes V(\lambda + \mu) \simeq V(\lambda)^* \otimes V(\lambda)\]
	for $\lambda \in P_+$.

	For any $U_q(\fg)$-module $V$, we denote the isotypical component of $V(\lambda)$ by $V^\lambda$ and the weight space of weight $\mu$ by $V_\mu$. Let $[V:V(\lambda)]$ be the multiplicity of $V(\lambda)$ in $V$.
	\subsection{Quantum coordinate algebra}
	Let $A,B$ be Hopf algebras. A {\it skew pairing} between $A$ and $B$ is a map
	\[A \times B \to {\mathbb C}\]
	such that
	\[(ab,c) = (a \otimes b, \Delta_B(c)),\]
	\[(a,cd) = (\Delta_A(a), d \otimes c),\]
	\[(1,c) = \varepsilon_A(c),\]
	\[(a,1) = \varepsilon_B(a),\]
	for $a,b \in A$, $c,d \in B$.

	If $A,B$ are Hopf $*$-algebras, we also assume
	\[(a^*,b) = \overline{(a,S(b)^*)}.\]

	For a pair of Hopf algebras with a skew pairing, one defines the following actions:
	for $a \in A$ and $b\in B$
	\begin{align*}
	a \triangleright b & := (a,b_{(2)}) b_{(1)}, & b \triangleleft a & := (a,b_{(1)}) b_{(2)}, \\
	b \triangleright a & := (a_{(1)},b) a_{(2)}, & a \triangleleft b & := (a_{(2)},b) a_{(1)}.
	\end{align*}
	Here we used the sumless Sweedler notation:
	\[\Delta(x) = x_{(1)} \otimes x_{(2)}.\]
	\begin{df}
	Let $\cO(G_q) \subset U_q(\fg)^*$ be the subspace of matrix coefficients of type 1 representations. Then $\cO(G_q)$ carries a unique Hopf $*$-algebra structure which makes the pairing $\cO(G_q) \times U_q(\fg) \to {\mathbb C}$ skew.

	Now for each type 1 module $V$, $v \in V$ and $l \in V^*$,
	\[c^V_{l,v}(x) := (xv,l)\]
	defines an element in $\cO(G_q)$. In the case of $V = V(\lambda)$, we shall write $c^\lambda_{l,v}$ instead of $c^{V(\lambda)}_{l,v}$.
	 Then the map
	\[\bigoplus_{\lambda \in P_+} V(\lambda) \otimes V(\lambda)^* \to \cO(G_q): v \otimes l \mapsto c^\lambda_{l,v}\]
	is an isomorphism of $U_q(\fg)$-modules.

	Let $\cO(T) := \cO(G_q)/\Ann(U_q(\fh))$. Then $\cO(T)$ can be identified with the algebra of regular functions on the maximal torus of $G$. Denote the canonical surjection $\cO(G_q) \to \cO(T)$ by $\pi_T$.

	Let $\Theta$ be the isomorphism
	\[\Theta:{\mathbb C}[K_{2\lambda} \mid \lambda \in P] \to \cO(T)\]
	defined by
	\[(\Theta(K_{2\lambda}),K_\mu) = q^{(\lambda,\mu)}.\]
	\end{df}
	\begin{df}
	Let $\displaystyle \cU(G_q) := \prod_{\lambda \in P_+} {\rm End}(V(\lambda))$ be the full dual of $\cO(G_q)$ and $\displaystyle c_c(\dG_q) := \bigoplus_{\lambda \in P_+} {\rm End}(V(\lambda)) \subset \cU(G_q)$. Then one can embed $U_q(\fg)$ into $\cU(G_q)$ and $c_c(\dG_q)$ is an ideal of $\cU(G_q)$.
	
	One can easily show that there is a one-to-one correspondence among
	\begin{enumerate}[{\rm (1)}]
	\item type 1 representations of $U_q(\fg)$,
	\item nondegenerate representations of $c_c(\dG_q)$ and
	\item continuous representations of $\cU(\dG_q)$.
	\end{enumerate}
	\end{df}
	\begin{rem}
	For any $\nu \in \fh^*$, $K_\nu$ makes sense as an element in $\cU(G_q)$ by the formula
	\[K_\nu v = q^{(\nu,{\rm wt}(v))}v\]
	for each weight vector $v \in V(\lambda)$. Then we again have
	\[K_\nu K_\mu = K_{\nu + \mu}\]
	for any $\nu,\mu \in \fh^*$. Moreover $K_{2\pi i \log(q)^{-1} \mu} = 1$
	for any $\mu \in Q^\vee$ shows $K_\nu$ actually makes sense for any $\nu \in X := \fh^*/2\pi i \log(q)^{-1}Q^\vee$. Then elements in $X$ are in a one-to-one correspondence with $1$-dimensional representations on $\cO(T)$.
	The character $K_\nu$ is a $*$-character if and only if $\nu \in i \fh^*_{\mathbb R}$.
	Notice that the Weyl group $W$ acts on $X$ in a natural way.
	\end{rem}
	Let us define the following central projections on $\cO(G_q)$
	\[p^\lambda := 1_\lambda \in \End(V(\lambda)) \subset c_c(G_q)\]
	and let $\varphi := p^0$ be the Haar state. For a type 1 $U_q(\fg)$-module $V$, $p^\lambda$ is nothing but the projection onto $V^\lambda$.

	We have
	\[\varphi \omega = \omega \varphi = \omega(1) \varphi\]
	for any $\omega \in \cU(\dG_q)$.

	(The universal C*-completion of) $\cO(G_q)$ is a compact quantum group in the sense of \cite{Wor}. In our notation, the modular automorphism of $\cO(G_q)$ is given by
	\[\sigma_t(x) = K_{-2it\rho} \triangleright x \triangleleft K_{-2it\rho},\]
	where $\rho$ is the half sum of positive roots: $\displaystyle \rho = \frac{1}{2} \sum_{\alpha \in \Delta_+} \alpha$.
	We also have
	\[S^2(x) = K_{-2\rho} \triangleright x \triangleleft K_{2\rho}.\]
	In particular,
	\[\sigma_i(x) = K_{2\rho} \triangleright x \triangleleft K_{2\rho} = S^2(K_{4\rho} \triangleright x)\]
	and hence
	\[\varphi(yx) = \varphi(S^2(K_{4\rho} \triangleright x)y),\]
	which can be rewritten as
	\[x \triangleright \varphi = \varphi \triangleleft (S^2(K_{4\rho} \triangleright x)).\]
	\subsection{Adjoint actions}\label{sec-ad}
	Let us recall some of the results which appeared in \cite{Bau} and \cite{Jos}.
	(Notice that their results are on the field $k(q)$, but the same proofs work for a fixed parameter $q \in {\mathbb C}^\times$ as long as $q$ is not a root of unity.)

	Recall the adjoint action of $U_q(\fg)$ on itself
	\[\ad(\omega) \mu := \omega_{(1)} \mu \dS(\omega_{(2)})\]
	and the coadjoint action on $\cO(G_q)$
	\[\ad(\omega) x := \omega_{(2)} \triangleright x \triangleleft \dS(\omega_{(1)}).\]
	Let $R \in \cU(G_q \times G_q) := \prod_{\lambda,\mu \in P_+} {\rm End}(V(\lambda) \otimes V(\mu))$ be the $R$-matrix. Namely, $R$ is given by the formula
	\[R := q^{\sum_{\alpha,\beta \in \Pi} (B^{-1})_{\alpha,\beta} H_\alpha \otimes H_\beta} \prod_{\alpha \in \Delta_+} \exp_{q_\alpha} ((1-q_\alpha^{-2}) F_\alpha \otimes E_\alpha).\]
	Here $B$ is the matrix $((\alpha^\vee,\beta^\vee))_{\alpha,\beta}$, $H_\alpha$ is the self-adjoint element in $\cU(G_q)$ which satisfies $q_\alpha^{H_\alpha} = K_\alpha$, $E_\alpha, F_\alpha$ are the PBW basis corresponding to $\alpha$ and
	\[\exp_q(x) := \sum_{n=0}^\infty q^{n(n+1)/2} \frac{x^n}{n_q !}.\]
	For our purpose, we do not need the whole formula, but the fact that $R$ is a sum of elements in $U_q(\fb^-) \otimes U_q(\fb^+)$ and
	\[({\rm id} \otimes \cP)(R) = (\cP \otimes {\rm id})(R) = q^{\sum_{\alpha,\beta \in \Pi} (B^{-1})_{\alpha,\beta} H_\alpha \otimes H_\beta}.\]

	Define $I: \cO(G_q) \to U_q(\fg)$ by
	\[I(x) := (x \otimes {\rm id})(R_{21} R_{12}).\]
	\begin{lem}\label{thm Bau}\cite[Theorem 3]{Bau} We have the following.
	\begin{enumerate}[{\rm (i)}]
	\item The map $I$ is an injective $U_q(\fg)$-module homomorphism.
	\item $I(zx) = I(z)I(x)$ for $z \in \cO(\CGq)$ and $x \in \cO(G_q)$.
	\item $\cP \circ I = \Theta^{-1} \circ \pi_T$.
	\end{enumerate}
	\end{lem}
	\begin{proof}
	For {\rm (i)} and {\rm (ii)}, see \cite{Bau}.

	For {\rm (iii)}, notice that
	\[q^{\sum_{\alpha,\beta \in \Pi} (B^{-1})_{\alpha,\beta} H_\alpha \otimes H_\beta} (v \otimes w) = q^{({\rm wt}(v),{\rm wt}(w))} (v \otimes w)\]
	shows
	\[(c^\lambda_{l,v} \otimes {\rm id})(q^{\sum_{\alpha,\beta \in \Pi} (B^{-1})_{\alpha,\beta} H_\alpha \otimes H_\beta}) = (l,v) K_{{\rm wt}(v)}.\]

	Here, by the definition of $\cP$, we have
	\[\cP(ab) = \cP(a) \cP(b)\]
	for $a \in U_q(\fb^-)$ and $b \in U_q(\fb^+)$.
	Hence
	\begin{align*}
	(\cP \circ I)(c^\lambda_{l,v}) & = \cP(c^\lambda_{l,v} \otimes {\rm id})(R_{21} R_{12}) \\
	& = (c^\lambda_{l,v} \otimes {\rm id})(q^{\sum_{\alpha,\beta \in \Pi} 2(B^{-1})_{\alpha,\beta} H_\alpha \otimes H_\beta}) \\
	& = (v,l) K_{2 {\rm wt}(v)} \\
	& = (\Theta^{-1} \circ \pi_T)(c^\lambda_{l,v}).
	\end{align*}
	\end{proof}
	Set $F(U_q(\fg)) := I(\cO(G_q))$ and let $Z$ be the center of $F(U_q(\fg))$. Now let us recall the separation theorem in \cite{Jos}.
	\begin{thm}
	There exists an adjoint invariant subspace $\bH \subset F(U_q(\fg))$ such that the multiplication map
	\[Z \otimes \bH \to F(U_q(\fg))\]
	is an isomorphism.
	Moreover the multiplicity of $V(\lambda)$ in $\bH$ is $\dim V(\lambda)_0$.
	\end{thm}
	Put $m := \dim(V(\lambda)_0)$.

	Decompose $\bH^\lambda$ into an $m$-fold direct sum of $V(\lambda)$'s. Fix a basis $(e_i)_{i=1}^{\dim V(\lambda)}$ of $V(\lambda)$ such that $e_i \in V(\lambda)_0$ for $1 \leq i \leq m$. 
	Take the corresponding basis $(a_{ij})_{i = 1}^m$ of the $j$-th copy of $V(\lambda)$ inside $\bH$.

	Now we have an $m \times m$-matrix $(a_{ij})$ with coefficients in $F(U_q(\fg))$. The determinant of $\cP_\lambda := \cP(a_{ij})$ is called the {\it quantum PRV determinant} and computed in \cite[Theorem 8.2.10]{Jos}.
	\begin{thm} We have
	\[\det \cP_\lambda = \prod_{n \in {\mathbb Z}_+, \alpha \in \Delta_+} (K_\alpha - q_\alpha^{2(n - (\rho,\alpha^\vee))} K_{-\alpha})^{\dim V(\lambda)_{n \alpha}}\]
	up to constant multiplication in ${\mathbb C}^\times$.
	\end{thm}
	From this, we can estimate the rank of $\cP_\lambda(\nu)$ for certain cases. (See \cite[Lemma 8.2.7]{Jos}). We regard $\cP_\lambda$ as a function on $\fh^*$ via $K_\alpha(\nu) = q^{(\mu,\alpha)}$.
	\begin{cor}\label{thm PRV}
	Let $\nu \in \fh^*$.
	\begin{enumerate}[{\rm (i)}]
	\item $\mathop{\rm rank} \cP_\lambda(\nu) = [F(U_q(\fg))/\Ann(V(\nu)) : V(\lambda)]$.
	\item The matrix $\cP_\lambda(\nu)$ is invertible if and only if  $(\nu+\rho,\alpha^\vee) \not \in {\mathbb Z}_+ + \daZ$ for any $\alpha \in \Delta_+$.
	\end{enumerate}
	\end{cor}
	As we shall see later, the twisted adjoint action
	\[{\rm ad}^S(a) x := (S \circ{\rm ad}(a) \circ S^{-1})(x) = a_{(1)} \triangleright x \triangleleft \dS^{-1}(a_{(2)})\]
	is more relevant to the Drinfeld double construction.
	We call the vector space of fixed points the {\it $q$-character algebra} and denote it by $\cO(\CqGq)$. Then this is an algebra and
	\[\chi_q(\lambda) := \sum_i (K_{-2\rho}v_i,l_i) c^\lambda_{l_i,v_i}\]
	forms a basis of $\cO(\CqGq)$, where $(v_i)$ is a basis of $V(\lambda)$ and $(l_i)$ is the dual basis. One can show $\cO(\CqGq)$ is isomorphic to the usual character algebra of $G$ as an algebra. In particular, this is commutative and its character space is $X/W$:
	any character on $\cO(\CqGq)$ is of the form $K_{\nu+2\rho}$ and these gives a same character if and only if $\nu$'s are in the same Weyl group orbit.
	\subsection{Drinfeld doubles}
	In this section, we collect some definitions and facts on Drinfeld doubles of $\cO(G_q)$.
	\begin{df}
	For Hopf algebras $A$ and $B$ with a skew pairing, the {\it Drinfeld double} $B \bowtie A$ is the algebra generated by $A$ and $B$ with the commutation relation
	\[ab = (a_{(1)} \triangleright b \triangleleft S(a_{(3)})) a_{(2)}\]
	for $a \in A$ and $b \in B$.
	As a vector space, the multiplication map gives an isomorphism $B \otimes A \to B \bowtie A$.

	If both $A$ and $B$ are Hopf $*$-algebras, $B \bowtie A$ is again a Hopf $*$-algebra.
	\end{df}
	\begin{rem}
	It is not necessary for $B$ to be a ``genuine'' Hopf algebra to define the Drinfeld double $B \bowtie A$ as an algebra, as long as the bimodule action of $A$ on $B$ makes sense.
	For example, one can define $D_c := c_c(\dG_q) \bowtie \cO(G_q)$ and $\tilde D := \cU(G_q) \bowtie \cO(G_q)$ in the same manner.
	\end{rem}

	Let $U_q(\fb^+)$ (resp.\ $U_q(\fb^-)$) be the subalgebra of $U_q(\fg)$ generated by $E_\alpha$ and $K_\lambda$ (resp.\ $F_\alpha$ and $K_\lambda$).

	Recall that there exists a unique skew-pairing
	\[U_q(\fb^-) \times U_q(\fb^+) \to \mathbb{C}\]
	such that
	\[(K_\lambda, K_\mu) = q^{(\lambda,\mu)},\]
	\[(F_\alpha,E_\beta) = -\frac{\delta_{\alpha,\beta}}{q_\alpha - q_\alpha^{-1}}.\]
	Consider the following skew-pairing
	\[U_q(\fb^-) \otimes U_q(\fb^+) \times U_q(\fb^+) \otimes U_q(\fb^-) \to \mathbb{C} : (a \otimes b, c \otimes d) = (a,c) (\dS(d),b).\]
	Now one can embed $\cO(G_q)$ into $U_q(\fb^-) \otimes U_q(\fb^+)$ as follows.
	\begin{prop}\cite[Lemma 9.2.13]{Jos}
	There is an algebra embedding
	$\Psi: \cO(G_q) \to U_q(\fb^-) \otimes U_q(\fb^+)$ such that
	\[(\Psi(x),a \otimes b) = (x,ba).\]
	\end{prop}
	For the later use, we prepare a technical lemma.
	\begin{lem}\label{thm PPsi}
	For $U_q(\fg)$-modules $V,W$, $\lambda,\mu \in P$, $v \in V_\lambda$ and $w \in W_\mu$,
	\[(\cP \otimes \cP)(\Psi(x)) (v \otimes w) = K_{\lambda - \mu}(x) v \otimes w.\]
	\end{lem}
	\begin{proof}
	By definition,
	\begin{align*}
	(\cP \otimes \cP)(\Psi(x)) (v \otimes w) & = ((\cP \otimes \cP)(\Psi(x)),K_\lambda \otimes K_\mu) (v \otimes w) \\
	& = (x, K_{\lambda-\mu}) (v \otimes w).
	\end{align*}
	\end{proof}
	The following result is  first observed by Kr\"{a}hmer \cite{Kra}.
	\begin{thm}
	The map
	\[\dD \times \Psi :U_q(\fg) \bowtie \cO(G_q) \to U_q(\fg) \otimes U_q(\fg)\]
	is injective.
	\end{thm}
	\subsection{The quantum group $SU_q(2)$}
	Through this section, let $\fg = \frsl_2$.
	In this case, the Cartan subalgebra $\fh$ is $1$-dimensional.
	We identify $\fh = {\mathbb C}$ with $\Pi = \{1\}$. Then $Q = {\mathbb Z}$ and $P = \frac{1}{2}{\mathbb Z}$.
	The Clebsch-Gordan formula asserts
	\[V(s) \otimes V(t) \simeq V(s+t) \oplus V(s+t-1) \oplus \dots \oplus V(|s-t|).\]

	Fix an orthonormal basis $(\xi_{\pm 1/2})$ of $V(1/2)$.
	Then we can fix generators $a,b,c,d$ of $\cO(SU_q(2))$ by
	\begin{align*}
	(a,x) & := (\pi^{1/2}(x)\xi_{1/2},\xi_{1/2}), & (b,x) &:= q(\pi^{1/2}(x)\xi_{1/2}, \xi_{-1/2}), \\
	(c,x) &:= q^{-1}(\pi^{1/2}(x)\xi_{-1/2},\xi_{1/2}), & (d,x) &:= (\pi^{1/2}(x)\xi_{-1/2},\xi_{-1/2}).
	\end{align*}
	Then
	\[ab = qba, ac = qca, bc = cb, ad - qbc = da - q^{-1}bc = 1,\]
	\[\Delta(a) = a \otimes a + b \otimes c, \Delta(b) = a \otimes b + b \otimes d, \Delta(c) = c \otimes a + d \otimes c, \Delta(d) = d \otimes d + c \otimes b.\]
	\section{Admissible representations}
	Let $D := U_q(\fg) \bowtie \cO(G_q)$. We regard $D$ and $D_c$ as subalgebras of $\tilde D$. Then $D_c$ is an ideal of $\tilde D$, hence in particular $D_c$ is a $D$-bimodule in a natural way.
	\begin{df}
	Let $V$ be a vector space and $\pi:D \to \End(V)$ a representation.
	We say $\pi$ is {\it admissible} if $V$ is of type 1 as a $U_q(\fg)$-module and $[V:V(\lambda)] < \infty$ for any $\lambda \in P_+$. We say $\pi$ is {\it spherical} if $\pi$ admits a nonzero $U_q(\fg)$-fixed vector.
	\end{df}
	The following lemma has already appeared in the proof of \cite[Theorem 8.1]{Kna}.
	\begin{lem}\label{thm subhomog}
	Let $A$ be a $*$-algebra and $N \in {\mathbb Z}_{\geq 0}$. Suppose $A$ is a subalgebra (with the $*$-structure ignored) of $\displaystyle \prod_{i \in I} \End(V_i)$, where $(V_i)_{i \in I}$ is a family of vector spaces with dimensions at most $N$. Then the dimension of any irreducible $*$-representation of $A$ is at most $N$.
	\end{lem}
	\begin{thm}\label{thm findim}
	Let $\pi$ be an irreducible $*$-representation of $D_c$ on a Hilbert space $H$. Then the multiplicity of $V(\lambda)$ in $\pi|_{c_c(\widehat{G_q})}$ is at most $\dim V(\lambda)$.
 	In particular, $\displaystyle V := \bigoplus_{\lambda \in P_+} \pi(p^\lambda) H$ is an irreducible admissible $D$-module.
	\end{thm}
	\begin{proof}
	For each $\mu = (\mu_1,\mu_2) \in P_+ \times P_+$, one can define a finite dimensional representation $\pi^\mu$ of $D$ by
	\[\pi^\mu = (\pi^{\mu_1} \otimes \pi^{\mu_2}) (\dD \times \Psi).\]
	Then since $\dD \times \Psi$ is injective, we get an embedding
	\[\bigoplus_{\mu \in P_+ \times P_+} \pi^\mu: D \hookrightarrow \prod_{\mu \in P_+ \times P_+} \End(V(\mu_1) \otimes V(\mu_2)).\]

	Fix $\lambda \in P_+$. By cutting the embedding above by $p^\lambda$, we get an embedding
	\[p^\lambda D_c p^\lambda \hookrightarrow \prod_{\mu \in P_+ \times P_+} \End(\pi^\mu(p^\lambda)(V(\mu_1) \otimes V(\mu_2))).\]
	Let $v_{w_0 \mu_2}$ be the lowest weight vector in $V(\mu_2)$.
	Since $v_{\mu_1} \otimes v_{w_0 \mu_2}$ is cyclic for the diagonal action of $U_q(\fg)$ on $V(\mu_1) \otimes V(\mu_2)$, the map
	\[{\rm Hom}_{U_q(\fg)}(V(\mu_1) \otimes V(\mu_2),V(\lambda)) \to V(\lambda): f \mapsto f(v_{\mu_1} \otimes v_{w_0 \mu_2})\]
	is injective.
	Hence we get
	\[[V(\mu_1) \otimes V(\mu_2) : V(\lambda)] \leq \dim V(\lambda).\]
	Therefore $\dim \pi^\mu(p^\lambda)(V(\mu_1) \otimes V(\mu_2)) \leq (\dim V(\lambda))^2$. Now we can apply Lemma \ref{thm subhomog} to get the desired conclusion.
	\end{proof}
	Now we start to classify admissible $D$-modules.
	\begin{lem}\label{thm classification of admissible} We have the following.
	\begin{enumerate}[{\rm (i)}]
	\item Let $V$ be an admissible $D$-module. Suppose $V^\lambda$ is an irreducible $p^\lambda D_c p^\lambda$-module. Then $V$ admits a unique irreducible subquotient containing a nonzero $\lambda$-isotypical component.
	\item Let $V,W$ be irreducible admissible $D$-modules. Suppose $0 \neq V^\lambda \simeq W^\lambda$ as $p^\lambda D_c p^\lambda$-modules. Then $V \simeq W$.
	\end{enumerate}
	\end{lem}
	\begin{proof}
	\begin{enumerate}[{\rm (i)}]
	\item Let $K := D V^\lambda$. We only need to show $K$ admits a unique maximal submodule. For any family of proper submodules $(L_i)_i$, $V^\lambda \cap L_i = \{0\}$ since any nonzero vector in $V^\lambda$ is cyclic in $K$. Then since
	\[L_i = \bigoplus_{0 \neq \lambda \in P_+} L_i^\lambda,\]
	$V^\lambda \cap \sum_i L_i = \{0\}$. Hence $\sum_i L_i$ is again a proper $D$-submodule. 
	\item Both $V$ and $W$ are quotients of $K := D_c p^\lambda \otimes_{p^\lambda D_c p^\lambda} V^\lambda \simeq D_c p^\lambda \otimes_{p^\lambda D_c p^\lambda} W^\lambda$. However since $K^\lambda = p^\lambda D_c p^\lambda \otimes_{p^\lambda D_c p^\lambda} V^\lambda \simeq V^\lambda$ is irreducible, $K$ has a unique irreducible subquotient containing a nonzero $\lambda$-isotypical component by {\rm (i)}. Hence $V \simeq W$.
	\end{enumerate}
	\end{proof}
	Let us restrict ourselves to the spherical cases. De Commer \cite{DCo} pointed out $\varphi D_c \varphi$ is actually isomorphic to the character algebra of $G$ for a general compact quantum group $G$. Let us state the result only in the case we need here.

	First, $D_c \varphi$ admits a $U_q(\fg)$-module structure by left multiplications.
	Since $c_c(\dG_q) \varphi = {\mathbb C} \varphi$, we have
	\[D_c \varphi = \cO(G_q) \varphi.\]
	Hence we get a $U_q(\fg)$-module structure on $\cO(G_q)$, but 
	\[x a \varphi = a_{(2)} (S(a_{(3)}) \triangleright x \triangleleft a_{(1)}) \varphi = a_{(2)} (x,S(a_{(3)}) a_{(1)}) \varphi = {\rm ad}^S(x)(a) \varphi\]
	shows this is nothing but the twisted adjoint action ${\rm ad}^S$.

	Hence, $a \varphi \in p^\lambda D_c \varphi$ if and only if $a$ is in the $\lambda$-isotypical component with respect to ${\rm ad}^S$. In particular, we get the following.
	\begin{lem}\label{thm CqGq}
	Recall $X = \fh^*/2\pi i \log(q)^{-1} Q^\vee$.
	The map
	\[\cO(\CqGq) \to \varphi D_c \varphi: x \mapsto x \varphi\]
	is an algebra isomorphism.
	In particular, $\varphi D_c \varphi$ is commutative and its character space is $X/W$.
	\end{lem}
	\begin{proof}
	Any element of $\varphi D_c \varphi$ can be written in the form $a \varphi$ for $a \in \cO(G_q)$. We also know that $a \in \varphi D_c \varphi$ if and only if $a$ is ${\rm ad}^S$-invariant, that is, $a \in \cO(\CqGq)$.
	
	We have
	\[a\varphi b \varphi = ab \varphi\]
	for $a \in \cO(G_q)$ and $b \in \cO(\CqGq)$, since $\varphi b \varphi = b \varphi$. In particular, this is an algebra isomorphism.
	\end{proof}
%	\begin{lem}
%	Any homomorphism $\cO({\rm Char}(G_q)) \to {\mathbb C}$ extends to a homomorphism $\cO(T) \to {\mathbb C}$.
%	\end{lem}
	\begin{cor}\label{thm X/W}
	Irreducible admissible spherical representations are parametrized by $X/W$.
	\end{cor}
	\begin{proof}
	Let $V$ be an irreducible admissible spherical representation. Then $V^0$ is an irreducible $\varphi D_c \varphi$-module, which is in a one-to-one correspondence with elements of $X/W$ by Lemma \ref{thm CqGq}.

	Thanks to Lemma \ref{thm classification of admissible}, two irreducible admissible spherical representations are isomorphic if and only if the corresponding characters on $\varphi D_c \varphi$ are equal, hence they give the same element of $X/W$.
	\end{proof}
\section{Parabolic inductions}\label{sec non-min}
	In this section, we give another construction of irreducible admissible spherical representations corresponding to $\nu \in X/W$, which can be considered as an analogue of parabolic inductions.

	Fix a subset $\Sigma \subset \Pi$ and let $(\fh^\Sigma)^*$ be the linear span of $\Sigma$. Then $\Sigma$ can be regarded as the set of simple roots of a Lie subalgebra $\fg^\Sigma \subset \fg$. Take a short root $\alpha$ in $\Sigma$ and set $q^\Sigma := q^{(\alpha,\alpha)/2}$. Let $P^\Sigma$ be the weight lattice of corresponding to $\Sigma$. Let $U_{q^\Sigma}(\fg^\Sigma)$ be the subalgebra of $U_q(\fg)$ generated by $E_\alpha,F_\alpha,K_\lambda$'s where $\alpha \in \Sigma$ and $\lambda \in P^\Sigma$. Then we have a quotient map $\pi^\Sigma: \cO(G_q) \to \cO(G^\Sigma_{q^\Sigma})$. We write elements associated to $\fg^\Sigma$ with superscript $\Sigma$, for example, $D^\Sigma$, $\rho^\Sigma$ etc.
%	Let $T^0$ be the maximal torus of $G^0_{q^0}$.
	For each $\nu \in X$, one may decompose $\nu$ to an orthogonal sum
	\[\nu = \nu^\Sigma + \nu^{\perp \Sigma},\]
	where $\nu^\Sigma \in (\fh^\Sigma)^*$ and $\nu^{\perp \Sigma} \perp (\fh^\Sigma)^*$.
	Then $\rho^\Sigma$ is nothing but the half sum of positive roots in $\Delta^\Sigma$, which is compatible with our previous notation.
	\begin{lem}
	Set $B_\Sigma := (U_{q^\Sigma}(\fg^\Sigma) U_q(\fh)) \bowtie \cO(G_q) \subset D$. Take $\nu \in \fh^*$ such that $\nu \perp \alpha$ for any $\alpha \in \Sigma$.
	Then for each admissible $D^\Sigma := U_{q^\Sigma}(\fg^\Sigma) \bowtie \cO(G^\Sigma _{q^\Sigma})$-module $V$, one can define a $B_\Sigma$-module structure on $V$ by
	\begin{itemize}
	\item For $x \in \cO(G_q)$,
	\[xv := \pi^\Sigma(x \triangleleft K_\nu) v,\]
	\item For $a \in U_{q^\Sigma}(\fg^\Sigma)$,
	\[av := av,\]
	\item For $\lambda \in P$,
	\[K_\lambda v := q^{({\rm wt} (v), \lambda)} v,\]
	when we regard ${\rm wt}(v) \in P^\Sigma \subset P$.
	\end{itemize}
	\end{lem}
	\begin{proof}
	We check each commutation relations. It is easy to show the above formula gives a $(U_{q^\Sigma}(\fg^\Sigma) U_q(\fh))$-module structure and an $\cO(G_q)$-module structure.
	Therefore, we only need to examine the commutation relation for $x \in \cO(G_q)$ and $a \in U_{q^\Sigma}(\fg^\Sigma) U_q(\fh)$.

	First, for $a \in U_{q^\Sigma}(\fg^\Sigma)$, notice that $a$ commutes with $K_\nu$. Hence
	\[xav = \pi^\Sigma(x \triangleleft K_\nu) a v = a_{(2)} \pi^\Sigma(\dS^{-1}(a_{(1)}) \triangleright x \triangleleft K_\nu a_{(3)}) v = a_{(2)}(\dS^{-1}(a_{(1)}) \triangleright x \triangleleft a_{(3)}) v.\]

	On the other hand, to see the commutation relation with $K_\lambda$ for $\lambda \in P$, let us remark that
	\[K_\lambda v = K_{\lambda^\Sigma} v\]
	by definition.
	We also have
	$K_{-\lambda^\Sigma} \triangleright \pi^\Sigma(x) \triangleleft K_{\lambda^\Sigma} = \pi^\Sigma(K_{-\lambda} \triangleright x \triangleleft K_{\lambda})$
	since $\lambda^{\perp \Sigma}$ commutes with $a \in U_{q^\Sigma}(\fg^\Sigma)$.
	Hence the above calculation also shows
	\[x K_\lambda v = K_\lambda (K_{-\lambda} \triangleright x \triangleleft K_\lambda) v,\]
	which is the desired relation.
	\end{proof}
	We denote the $B_\Sigma$-module given in the lemma above by $V_{(0,\nu)}$.
	\begin{rem}
	The $0$ is used to show that we work with the spherical cases. One can also define $V_{(\lambda,\nu)}$ to get parabolic inductions in the nonspherical cases in a similar way, but in this paper, we do not treat representations of this type.
	\end{rem}

	Now we define a quantum analogue of parabolic inductions. For an admissible $D^\Sigma$-module $V$, define a $D$-module ${\rm Ind}_\Sigma^\Pi (V,\nu)$ by
	\[{\rm Ind}_\Sigma^\Pi (V,\nu) := D_c \otimes_{B_\Sigma} V_{(0,\nu-2\rho^{\perp \Sigma})}.\]
	
	In the case of $\Sigma = \emptyset$, $B := B_\emptyset$ is $U_q(\fh) \bowtie \cO(G_q)$. Since $D^\emptyset = {\mathbb C}$, $\mathbb C$ admits a unique $D^\emptyset$-module structure.
	Put
	\[L(0,\nu) := {\rm Ind}^\Pi_\emptyset ({\mathbb C},\nu) = D_c \otimes_B {\mathbb C}_{(0,\nu-2\rho)}.\]
	Let $\Lambda$ be the map
	\[c_c(\dG_q) \to L(0,\nu) : \omega \mapsto \omega \otimes 1.\]
	Then $\Lambda$ gives a $U_q(\fg)$-module isomorphism
	\[\{\omega \in c_c(\dG_q) \mid \omega K_\lambda = \omega, \forall \lambda\} \to L(0,\nu).\]
	In particular, all $L(0,\nu)$ are isomorphic to the same $U_q(\fg)$-module $L = {\rm Im}(\Lambda)$ as $U_q(\fg)$-module and $[L(0,\nu):V(\lambda)] = \dim V(\lambda)_0$.
	It is often convenient to think $L(0,\nu)$ as a family of representations $\pi^\nu$ on the same vector space $L$. Then by definition, for $\omega \in c_c(\dG_q)$ and $x \in \cO(G_q)$,
	\[\pi^\nu(x) \Lambda(\omega) = K_{\nu-2\rho}(x_{(2)}) \Lambda \left(x_{(3)} \triangleright \omega \triangleleft S(x_{(1)})\right).\]
	The importance of the module $L(0,\nu)$ is as follows.
	\begin{prop}\label{thm HC}
	We have the following.
	\begin{enumerate}[{\rm (i)}]
	\item The $D$-module $L(0,\nu)$ admits a unique irreducible spherical subquotient. We denote it by $V(0,\nu)$.
	\item Any irreducible admissible spherical $D$-module isomorphic to one of $V(0,\nu)$.
	\item $V(0,\nu) \simeq V(0,\nu')$ if and only if $\nu' \in W\nu$.
	\end{enumerate}
	\end{prop}
	\begin{proof}
	Thanks to Lemma \ref{thm classification of admissible}, $L(0,\nu)$ admits a unique irreducible spherical subquotient.

	By \cite[Remark 32]{DFY}, notice that
	\[x \varphi = \varphi \sigma_i(x) = \varphi (K_{2\rho} \triangleright x \triangleleft K_{2\rho})\]
	for $x \in \cO(\CqGq)$. Hence
	\[x \Lambda(\varphi) = K_{\nu+2\rho}(x) \Lambda(\varphi).\]
	Hence $V(0,\nu)$ corresponds to $\nu \in X/W$ in the classification of Corollary \ref{thm X/W}.
	\end{proof}
	\begin{rem}
	As we shall see in Proposition \ref{thm rank}, with a suitable choice of $\nu$ in the same Weyl group orbit, $V(0,\nu)$ is actually a submodule of $L(0,\nu)$.
	\end{rem}
	We conclude this section with the following ``induction-by-step'' type lemma.
	Notice that for each $\nu \in X$, $L(0,\nu^\Sigma)$ and $V(0,\nu^\Sigma)$ are $D^\Sigma$-modules.
	\begin{lem}
	We have an isomorphism
	\[L(0,\nu) \simeq {\rm Ind}_\Sigma^\Pi (L(0,\nu^\Sigma), \nu^{\perp \Sigma}).\]
	In particular, the module ${\rm Ind}_\Sigma^\Pi (V(0,\nu^\Sigma),\nu^{\perp \Sigma})$ is a spherical subquotient of $L(0,\nu)$.
	\end{lem}
	\begin{proof}
	By definition, we have
	\[{\rm Ind}_\Sigma^\Pi (L(0,\nu^\Sigma), \nu^{\perp \Sigma}) = D_c \otimes_{B_\Sigma} D^\Sigma_c \otimes_{B^\Sigma} {\mathbb C}_{(0,\nu^\Sigma - 2\rho^\Sigma)}.\]
	(Notice that $B^\Sigma = \cO(G^\Sigma_{q^\Sigma}) \bowtie U_{q^\Sigma}(\fh^\Sigma) \neq B_\Sigma$.)
	We claim
	\[D_c \otimes_{B_\Sigma} D^\Sigma_c \otimes_{B^\Sigma} {\mathbb C}_{(0,\nu^\Sigma - 2\rho^\Sigma)} \to D_c \otimes_B {\mathbb C}_{(0,\nu-2\rho)}: \omega \otimes \mu \otimes 1 \mapsto \omega \mu \otimes 1\]
	is an isomorphism.

%	It is clear that this defines a surjective homomorphism.
	To construct the inverse, for each $\omega \in c_c(\dG_q)$, we can find an idempotent $\mu \in c_c(\dG^\Sigma_q)$ such that $\omega \mu = \omega$.
	Now one can define a map
	\[D_c \otimes_B {\mathbb C}_{(0,\nu-2\rho)} \to D_c \otimes_{B_\Sigma} D^\Sigma_c \otimes_{B^\Sigma} {\mathbb C}_{(0,\nu^\Sigma - 2\rho^\Sigma)}: \omega \otimes 1 \mapsto \omega \otimes \mu \otimes 1\]
	where $\mu$ is an idempotent in $c_c(\dG_q)$ such that $\omega \mu = \omega$.
	Here we notice it does not depend on the choice of $\mu$. In fact, for $\mu_1,\mu_2 \in c_c(\dG_q)$ with $\omega \mu_i = \omega$, one can find an idempotent $\mu_0$ such that $\mu_i \mu_0 = \mu_i$ for $i=1,2$. Then
	\[\omega \otimes \mu_1 = \omega \otimes \mu_1 \mu_0 = \omega \mu_1 \otimes \mu_0 = \omega \otimes \mu_0 = \omega \otimes \mu_2.\]
	Therefore this map is well-defined.
	
	These maps are inverses to each other, which finishes the proof.
	\end{proof}
	\section{Invariant form}
	In this section, we construct an invariant sesquilinear pairing on certain pairs of parabolic inductions. Before that, we begin with the following standard observation.

	Let $V,W$ be $D$-modules. We say a pairing $V \times W \to {\mathbb C}$ is {\it invariant} if
	\[(av,w) = (v,a^*w)\]
	for $a \in D$, $v \in V$ and $w \in W$.
	We say $V$ is {\it unitarizable} if it admits an invariant inner product.
	\begin{lem}\label{thm Schur}
	We have the following.
	\begin{enumerate}[{\rm (i)}]
	\item Let $V,W$ be admissible irreducible $D$-modules. Then an invariant sesquilinear pairing between $V$ and $W$ is unique up to a scalar factor, if it exists.
	\item Let $V,W_i$ be admissible irreducible $D$-modules with invariant sesquilinear pairings between $V$ and $W_i$ for $i = 1,2$. Then $W_1 \simeq W_2$.
	\end{enumerate}
	\end{lem}
	\begin{proof}
	For an admissible $D$-module $V$, consider the vector space ${\overline{V}}^*$ of all conjugate linear functionals on $V$. Then ${\overline{V}}^*$ carries a natural $D$-module structure by
	\[(ax,y) := (x,a^*y)\]
	for $x \in {\overline{V}}^*$, $y \in V$.

	Consider the submodule $V\tilde{} = \oplus_\lambda (V^\lambda)^* \subset  {\overline{V}}^*$.
	Then $V \tilde{}$ is admissible, and if $V$ is irreducible, then $V \tilde{}$ is also irreducible.
	
	Now sesquilinear pairings between $V$ and $W$ are in one-to-one correspondence to homomorphisms from $W$ to $V\tilde{}$. Now the results are direct consequences of Schur's lemma.
	\end{proof}

	Define a functional $\dphi \in c_c(\dG_q)$ by
	\[\dphi(x) = \sum_{\lambda \in P_+} {\rm Tr}_\lambda(K_{2\rho}) {\rm Tr}_\lambda(K_{-2 \rho} x),\]
	where ${\rm Tr}_\lambda$ is the non-normalized trace on $V(\lambda)$.
	In \cite{VD}, it is shown that this is the left invariant weight on $c_c(\dG_q)$, that is, a positive functional on $c_c(\dG_q)$ such that
%	\[\dphi(x_{(2)})x_{(1)} = \dphi(x) 1,\]
%	\[\dphi(x_{(1)})x_{(2)} = \dphi(x) K_{4\rho}.\]
%	We observe
	\[\dphi(x \triangleleft a) = \dphi(x) \varepsilon(a).\]
	We also have
	\[\dphi(a \triangleright x) = \dphi(x) K_{4\rho}(a).\]
	Therefore we get
	\[\dphi((a \triangleright x)y) = \dphi(a_{(2)} \triangleright (x (S^{-1}(a_{(1)}) \triangleright y))) = K_{4\rho}(a_{(2)}) \dphi(x(S^{-1}(a_{(1)})\triangleright y)),\]
	\[\dphi((x \triangleleft a)y) = \dphi(x(y \triangleleft S(a_{(1)})) \triangleleft a_{(2)}) = \dphi(x(y \triangleleft S(a))).\]
%	One can equip a natural inner product on $c_c(\dG_q)$ by
%	\[(x,y) := \dphi(y^*x)\]
%	\[\dphi(x_{(2)})x_{(1)} = \dphi(x) 1,\]
%	\[\dphi(x_{(1)})x_{(2)} = \dphi(x) K_{4\rho}.\]
%	Hence
%	\[\dphi(a \triangleright x) = \dphi(x) \varepsilon(a),\]
%	\[\dphi(x \triangleleft a) = \dphi(x) K_{4\rho}(a).\]
%	Therefore we get
%	\[\dphi((a \triangleright x)y) = \dphi(a_{(2)} \triangleright (x (S^{-1}(a_{(1)}) \triangleright y))) = \dphi(x(S^{-1}(a)\triangleright y)),\]
%	\[\dphi((x \triangleleft a)y) = \dphi(x(y \triangleleft S(a_{(1)})) \triangleleft a_{(2)}) = \dphi(x(y \triangleleft S(a_{(1)}))) K_{4\rho}(a_{(2)}).\]

%	Notice that $L$ carries a natural inner product
%	\[(x,y) := \dphi(y^* x).\]
%	Now restricting the sesquilinear pairing
%	\[(c_c(\dG_q) \otimes V) \times (c_c(\dG_q) \otimes V\tilde{}) \to {\mathbb C},\]
%	we get a sesquilinear pairing
%	\[{\rm Ind}_\Sigma^\Pi (V,\nu) \times {\rm Ind}_\Sigma^\Pi (V\tilde{},-\overline{\nu}) \to {\mathbb C}.\]

	Put $P_{\Sigma,+} := \{\lambda \in P \mid (\lambda,\alpha) \in {\mathbb Z}_{\geq 0} \text{ for } \alpha \in \Sigma \}$.
	In a similar way, one may consider irreducible representations $V(\lambda)$ of weight $\lambda \in P_{\Sigma,+}$ of the ``reductive'' quantum group $U_{q^\Sigma}(\fg^\Sigma) U_q(\fh)$ to construct a compact quantum group $\cO(G_{\Sigma,q})$.
%	Again we can define $\displaystyle \cU(G_{\Sigma,q}) = \bigoplus_{\lambda \in P_{\Sigma,+}} \End(V(\lambda))$.
	Again from \cite{VD}, the functional $\dphi_\Sigma$ on $\displaystyle c_c(\dG_{\Sigma,q}) := \bigoplus_{\lambda \in P_{\Sigma,+}} \End(V(\lambda))$.
	\[\dphi_\Sigma(x) = \sum_{\lambda \in P_{\Sigma.+}} {\rm Tr}(K_{2\rho^\Sigma}) {\rm Tr}(K_{-2\rho^\Sigma}x)\]
	satisfies the same relations as $\dphi$.

	Now for each $x \in c_c(\dG_q)$, one can find a unique $E(x) \in \cU(G_{\Sigma,q})$ such that 
	\[\dphi(xy) = \dphi_\Sigma(E(x)y).\]
	In this way, we can define a linear map $E:c_c(\dG_q) \to \cU(G_{\Sigma,q})$.

	The following lemma can be considered as a simple consequence of the modular theory, but for convenience, we attach a purely algebraic proof.
	\begin{lem}
	\begin{enumerate}[{\rm (i)}]
	\item The map $E$ is positive.
	\item The map $E$ is a $\cU(G_{\Sigma,q})$-bimodule homomorphism.
	\item We have
	\[E(a \triangleright x) = (a \triangleleft K_{4 \rho^{\perp \Sigma}}) \triangleright E(x),\]
	\[E(x \triangleleft a) = E(x) \triangleleft a.\]
	\end{enumerate}
	\end{lem}
	\begin{proof}
	First we show $E$ is positive. In fact, 
	\begin{align*}
	\dphi_\Sigma(y^* E(x^*x ) y) & = \dphi_\Sigma(E(x^*x)y K_{-2\rho^\Sigma}y^* K_{2\rho^\Sigma}) \\
	& = \dphi(x^* x y K_{-2\rho^\Sigma}y^* K_{2\rho^\Sigma}) \\
	& = \dphi(y^* x^*x y).
	\end{align*}
	Here for the last equation, we used the fact that $K_{2\rho^{\perp \Sigma}}$ commutes with $y$.

	For the second assertion, for $a \in \cU(G_{\Sigma,q})$, we have
	\[\dphi_\Sigma(E(ax)y) = \dphi(axy) = \dphi(xyK_{-2\rho}aK_{2\rho}) = \dphi_\Sigma(E(x)yK_{-2\rho^\Sigma}aK_{2\rho^\Sigma}) = \dphi_\Sigma(aE(x)y).\]
	Hence $E$ is a left $\cU(G_{\Sigma,q})$-module homomorphism. One can show $E$ is a right $\cU(G_{\Sigma,q})$-module homomorphism in a similar manner.

	For the third assertion,
	\begin{align*}
	\dphi_\Sigma(E(a \triangleright x)y) &= \dphi((a \triangleright x)y) = \dphi(x((S^{-1}(a \triangleleft K_{4\rho})) \triangleright y)) \\
	& = \dphi_\Sigma(E(x) ((S^{-1}(a \triangleleft K_{4\rho})) \triangleright y)) = \dphi_\Sigma(((a \triangleleft K_{4\rho^{\perp\Sigma}}) \triangleright E(x)) y).
	\end{align*}
	Again, $E(x \triangleleft a) = E(x) \triangleleft a$ is similar.
	\end{proof}
	Let $V,W$ be admissible $D$-modules with an invariant sesquilinear pairing. Thanks to {\rm (ii)} of the lemma above, one can define a sesquilinear pairing on ${\rm Ind}_\Sigma^\Pi (V,\nu)$ and ${\rm Ind}_\Sigma^\Pi (W,-\overline{\nu})$ by
	\[(x \otimes v, y \otimes w) := (E(y^*x)v,w),\]
	where ${\rm Ind}_\Sigma^\Pi (V,\nu)$ is identified with $c_c(\dG_q) \otimes_{U_{q^\Sigma}(\fg^\Sigma) U_q(\fh)} V$.
	If we start with $V = W$, then the resulting sesquilinear form is an inner product if and only if the original form on $V$ is positive definite.
	\begin{prop}
	This sesquilinear pairing is invariant, i.e.\ 
	\[(ax,y) = (x,a^*y)\]
	for any $a \in D$, $x \in {\rm Ind}_\Sigma^\Pi (V,\nu)$ and $y \in {\rm Ind}_\Sigma^\Pi (W,-\overline{\nu})$.
	\end{prop}
	\begin{proof}
	Trivial for $a \in U_q(\fg)$.

	The assertion follows for $a \in \cO(G_q)$ also by a calculation:
	\begin{align*}
	(x \otimes v, a^*(y \otimes w)) & = (x \otimes v, (a_{(3)}^* \triangleright y \triangleleft S(a_{(1)}^*)) \otimes \pi^\Sigma(a_{(2)}^* \triangleleft K_{-\overline{\nu}-2\rho^{\perp \Sigma}}) w) \\
	& = (E((a_{(3)}^* \triangleright y \triangleleft S(a_{(1)}^*))^* x) v,\pi^\Sigma(a_{(2)}^* \triangleleft K_{-\overline{\nu}-2\rho^{\perp\Sigma}})w) \\
	& = (\pi^\Sigma(a_{(2)} \triangleleft K_{\nu + 2\rho^{\perp\Sigma}})E((S(a_{(3)}) \triangleright y^* \triangleleft a_{(1)}) x)v,w) \\
	& = (((a_{(4)} \triangleleft K_{\nu+2\rho^{\perp \Sigma}}) \triangleright E((S(a_{(5)}) \triangleright y^* \triangleleft a_{(1)}) x) \triangleleft S(a_{(2)})) \pi^\Sigma(a_{(3)}) v,w) \\
	& = (E(y^*( (a_{(3)} \triangleleft K_{\nu-2\rho^{\perp \Sigma}}) \triangleright x \triangleleft S(a_{(1)}))) \pi^\Sigma(a_{(2)}) v,w) \\
	& = (E(y^*(a_{(3)} \triangleright x \triangleleft S(a_{(1)}))) \pi^\Sigma(a_{(2)} \triangleleft K_{\nu - 2\rho^{\perp \Sigma}}) v, w) \\
	& = ((S(a_{(3)}) \triangleright x \triangleleft a_{(1)}) \otimes \pi^\Sigma(a_{(2)} \triangleleft K_{\nu-2\rho^{\perp \Sigma}}) v, y \otimes w)\\
	& = (a(x \otimes v), y \otimes w).
	\end{align*}
	\end{proof}
	Since the invariant sesquilinear form is unique, we get the following.
	\begin{cor}\label{thm ind form}
	Let $\nu \in \fh^*$. Suppose $\nu$ satisfies the following three conditions.
	\begin{itemize}
	\item $V(0,\nu) = {\rm Ind}_\Sigma^\Pi (V(0,\nu^\Sigma),\nu^{\perp \Sigma})$,
	\item $V(0,\nu^\Sigma)$ admits a nonzero invariant sesquilinear form and
	\item $\nu^{\perp \Sigma}$ is imaginary.
	\end{itemize}
	Then $V(0,\nu)$ is unitarizable if and only if $V(0,\nu^\Sigma)$ is.
	\end{cor}
	In the case of $\Sigma = \emptyset$, the pairing is given by
	\[(\Lambda(x),\Lambda(y)) = \dphi(y^* x),\]
	hence is an inner product on $L$ which satisfies
	\[(\pi^\nu(a)v,w) = (v,\pi^{-\overline{\nu}}(a^*)w).\]

	To conclude this section, we state an application of this sesquilinear pairing which plays an essential role in the next section.

	Let $L^0(0,\nu)$ be the cyclic submodule of $L(0,\nu)$ generated by $\Lambda(\varphi)$.
	\begin{lem}\label{thm cyc-ann}
	We have
	\[V(0,\nu) = L^0(0,\nu)/\Ann L^0(0,-\overline{\nu}).\]
	In particular, if $\Lambda(\varphi)$ is cyclic both in $L(0,\nu)$ and $L(0,-\overline{\nu})$, then $L(0,\nu)$ is irreducible.
	\end{lem}
	\begin{proof}
	Set $L^{00}(0,\nu) := L^0(0,\nu)/\Ann L^0(0,-\overline{\nu})$. Then this is an admissible spherical $D$-module with a $D$-invariant nondegenerate pairing
	\[L^{00}(0,\nu) \times L^{00}(0,-\overline{\nu}) \to {\mathbb C}.\]
	Notice that $\Lambda(\varphi)$ is cyclic both in $L^{00}(0,\nu)$ and $L^{00}(0,-\overline{\nu})$.

	Take a $D$-submodule $K$ of $L^{00}(0,\nu)$. If $\Lambda(\varphi) \in K$, $K = L^{00}(0,\nu)$ since $\Lambda(\varphi)$ is cyclic. If $\Lambda(\varphi) \not \in K$, since the pairing is $U_q(\fg)$-invariant, $\Lambda(\varphi) \in \Ann K$. Hence $\Ann K = L^{00}(0,-\overline{\nu})$. Since the pairing is nondegenerate, $K = 0$. Therefore $L^{00}(0,\nu)$ is irreducible.
	\end{proof}
	A consequence of the lemma above is that $V(0,\nu)$ and $V(0,-\overline{\nu})$ admit an invariant sesquilinear pairing. In particular, together with Lemma \ref{thm Schur} and Proposition \ref{thm HC}, we get the following.
	\begin{cor}\label{thm sesqui-form}
	The $D$-module $V(0,\nu)$ admits a nonzero invariant sesquilinear form if and only if $-\overline{\nu} = w \nu$ for $w \in W$.
	\end{cor}
	\section{Irreducibility}
	In this section, we give a criterion for the irreducibility of $L(0,\nu)$. The results of this section deeply depend on those of Section \ref{sec-ad}.

	As in the remark before Lemma \ref{thm CqGq}, each elements in the $\lambda$-isotypical component of $\cO(G_q)$ with respect to ${\rm ad}^S$ transports $V^0$ to $V^\lambda$. In $L(0,\nu)$, we can say more,
	\begin{align*}
	x \Lambda(\varphi) & = K_{\nu-2\rho}(x_{(2)}) \Lambda(x_{(3)} \triangleright \varphi \triangleleft S(x_{(1)})) \\
	& = K_{\nu-2\rho}(x_{(2)}) \Lambda(\varphi \triangleleft (S^2(K_{4\rho} \triangleright x_{(3)})S(x_{(1)}))) \\
	& = K_{\nu+2\rho}(x_{(2)}) \Lambda(\varphi \triangleleft S(x_{(1)}S(x_{(3)}))).
	\end{align*}

	Notice that
	\[\Gamma:x \mapsto x_{(1)} S(x_{(3)}) \otimes x_{(2)}\]
	is a coaction of $\cO(G_q)$ on itself and
	\[(a \otimes {\rm id})\Gamma(x) = {\rm ad}^S(a)(x).\]

	Consider $H := S \circ I^{-1}(\bH)$. Due to Lemma \ref{thm Bau}, this is an ${\rm ad}^S$ invariant subspace of $\cO(G_q)$ such that the multiplication map
	\[H \otimes \cO(\CqGq) \to \cO(G_q)\]
	is isomorphic.

	Fix $\lambda \in P_+$ and put $m := \dim V(\lambda)_0$.
	Recall $(a_{ij}) \subset \bH$ in Section \ref{sec-ad}. Let $b_{ij} := S(I^{-1}(a_{ij}))$.
	Let $(f_i)$ be the dual basis of $(e_i)$ and $c^\lambda_{ij} := c^\lambda_{f_i,e_j}$.
	Then we have $\Gamma(b_{ij}) = \sum_k c^\lambda_{ik} \otimes b_{kj}$.
	Hence the calculation above shows
	\[b_{ij} \Lambda(\varphi) = \sum_{k=1}^m K_{\nu+2\rho}(b_{kj}) \Lambda(\varphi \triangleleft S(c^\lambda_{ik})).\]
	(For this, notice that $\Lambda(\varphi \triangleleft S(c^\lambda_{ij})) = 0$ unless $e_j$ is of weight $0$.)

	Thus the multiplicity of $V(\lambda)$ in $L^0(0,\nu)$ is the rank of the matrix $(K_{\nu+2\rho}(b_{ij}))_{i,j = 1}^m$.

	\begin{align*}
	(K_{\nu+2\rho}(b_{ij}))_{i,j} & = (K_{\nu+2\rho}(S(I^{-1}(a_{ij}))))_{i,j} \\
	& = (K_{-\nu-2\rho}(\pi_T(I^{-1}(a_{ij}))))_{i,j}\\
	& = (K_{-\nu-2\rho}(\Theta \circ \cP(a_{ij})))_{i,j} \\
	& = \cP_\lambda(-\nu/2 - \rho).
	\end{align*}
	Using this equality, we can reformulate Corollary \ref{thm PRV} in our setting.
	\begin{prop}\label{thm rank}
	Fix $\nu \in \fh^*$. We have the following.
	\begin{enumerate}[{\rm (i)}]
	\item $[L^0(0,\nu):V(\lambda)] = [F(U_q(\fg))/\Ann V(-\nu/2 - \rho): V(\lambda)]$.
	\item The spherical vector is cyclic in $L(0,\nu)$ if and only if $(\nu,\alpha^\vee) \not \in 2{\mathbb Z}_- + 2 \pi i \log(q)^{-1} {\mathbb Z}$ for any $\alpha \in \Delta_+$. In particular, any admissible spherical irreducible representation is a quotient of $L(0,\nu)$ for some $\nu \in \fh^*$.
	\item Any nonzero submodule of $L(0,\nu)$ contains the spherical vector if and only if $(\nu,\alpha^\vee) \not \in 2{\mathbb Z}_+ + 2 \pi i \log(q)^{-1} {\mathbb Z}$ for any $\alpha \in \Delta_+$. In this case, $L^0(0,\nu) = V(0,\nu)$. In particular, any admissible spherical irreducible representation is a subrepresentation of $L(0,\nu)$ for some $\nu \in \fh^*$. 
	\item The $D$-module $L(0,\nu)$ is irreducible if and only if $(\nu,\alpha^\vee) \not \in (2 \mathbb Z \setminus \{0\}) + 2 \daZ$ for any $\alpha \in \Delta_+$.
	\end{enumerate}
	\end{prop}
	\begin{proof}
	From the argument above together with (i) in Corollary \ref{thm PRV}, we get (i).

	For (ii), by definition, the spherical vector is cyclic if and only if $L^0(0,\nu) = L(0,\nu)$. On the other hand, since $[L(0,\nu):V(\lambda)] = \dim V(\lambda)_0 = m$, the argument above shows $L^0(0,\nu) = L(0,\nu)$ if and only if $\cP_\lambda(-\nu/2 - \rho)$ is invertible for any $\lambda \in P_+ \cap Q$. Combining with (ii) in Corollary \ref{thm PRV}, we get (ii).

	Part (iii) is equivalent to (ii) in view of the pairing between $L(0,\nu)$ and $L(0,-\overline{\nu})$. Part(iv) follows from Lemma \ref{thm cyc-ann}. 
	\end{proof}
	As corollaries, we get the irreducibility of certain $D$-modules we have constructed. First, we begin with the integral case.
	\begin{cor}\label{thm integral}
	For $\mu \in P_+ + \pi i \log(q)^{-1} P$,
	\[V(0,-2\mu - 2\rho) \simeq V(\mu) \otimes V(\mu)^*\]
	as a $D$-module, where the module structure of the right hand side is defined as in Theorem \ref{thm findim}, namely,
	\[a(v \otimes w) = a_{(1)}v \otimes a_{(2)} w,\]
	\[x (v \otimes w) = \Psi(x) (v \otimes w)\]
	for $a \in U_q(\fg)$, $x \in \cO(G_q)$.
	\end{cor}
	\begin{proof}
	First, the left hand side is a $U_q(\fg)$-module of type 1.

	We claim
	\[\Psi(x) v = K_{-2\mu}(x) v\]
	for $x \in \cO(\CqGq)$ and the unique (up to a scalar factor) $U_q(\fg)$-invariant vector $v$ in $V(\lambda) \otimes V(\lambda)^*$. Here since $x \in \cO(\CqGq)$ preserves $0$-isotypical component, we know $\Psi(x)v$ is a multiple of $v$. Therefore we only need to show
	\[(\Psi(x)v, l_\mu \otimes v_\mu) = K_{-2\mu}(x) (v,l_\mu \otimes v_\mu),\]
	where $l_\mu \in V(\mu)^*$ such that $l_\mu(v_\mu) = 1$, $l_\mu(v) = 0$ for any weight vector $v$ which is not of highest weight.
	Now since $\Psi(x)$ is in $U_q(\fb^+) \otimes U_q(\fb^-)$, the left hand side is nothing but
	\[((\cP \otimes \cP)(\Psi(x)) v,l_\mu \otimes v_\mu).\]
	Now Lemma \ref{thm PPsi} asserts this is equal to
	\[K_{-2\mu}(x) (v, l_\mu \otimes v_\mu),\]
	which shows the claim.

	Therefore from Lemma \ref{thm classification of admissible} and Proposition \ref{thm HC}, $V(\mu) \otimes V(\mu)^*$ contains $V(0,{-2\mu-2\rho})$ as a subquotient. Since $F(U_q(\fg))/\Ann V(\mu) \simeq \End(V(\mu))$, the first assertion of the proposition above shows
	\[[V(\mu) \otimes V(\mu)^*:V(\lambda)] = [L^0(0,{-2\mu-2\rho}):V(\lambda)]\]
	for any $\lambda \in P_+$. Using (iii) of Proposition \ref{thm rank}, we get
	\[V(0,-2\mu - 2\rho) = L^0(0,{-2\mu-2\rho}) \simeq V(\mu) \otimes V(\mu)^*.\]
	\end{proof}
	Consequently, we get a description of $V(0,\nu)$ in the case of $\fg = \frsl_2$.
	\begin{cor}
	Let $\fg := \frsl_2$. Take $\nu \in {\mathbb C} / 2 \pi i \log(q)^{-1} {\mathbb Z}$.
	\begin{enumerate}[{\rm (i)}]
	\item If $\nu \not \in ({\mathbb Z} \setminus \{0\}) + \pi i \log(q)^{-1} {\mathbb Z}$, $V(0,\nu) = L(0,\nu)$.
	\item If $\nu \in {\mathbb Z}_- + \pi i \log(q)^{-1} {\mathbb Z}$, $\displaystyle V(0,\nu) = V\left(\frac{-\nu-1}{2}\right) \otimes V\left(\frac{-\nu-1}{2}\right)^*.$
	\end{enumerate}
	\end{cor}
	We prepare an additional proposition.
	\begin{prop}\label{thm irrep ind}
	Let $\Sigma \subset \Pi$. Suppose that
	for any $\alpha \in \Delta$, if $(\nu,\alpha^\vee) \in 2{\mathbb Z} + 2 \daZ$, then $\alpha \in {\rm span} \Sigma$.
	Then
	\[V(0,\nu) = {\rm Ind}_\Sigma^\Pi (V(0,\nu^\Sigma),\nu^{\perp \Sigma}).\]
	\end{prop}
	\begin{proof}
	By conjugating by a element in $W^\Sigma$ if necessary, we may assume $\Re(\nu,\alpha^\vee) \geq 0$ for any $\alpha \in \Sigma$. 

	Let $w_0$ be the longest element of $W^\Sigma$.
	Notice that $\nu$ is in the case (ii) in Proposition \ref{thm rank}, while $w_0 \nu$ is in the case (iii).

	We first claim $\Hom_D(L(0,\nu),L(0,w_0 \nu))$ is $1$-dimensional.
	
	Take $T \in \Hom_D(L(0,\nu),L(0,w_0 \nu))$.
	Consider the kernel $K$ of the quotient map $L(0,\nu) \to V(0,\nu)$. Then $T K$ is a submodule of $L(0,\nu)$ which does not contain the spherical vector.  
	Thanks to Proposition \ref{thm rank}, $TK = 0$.
	Therefore $T$ factors
	\[L(0,\nu) \to V(0,\nu) \to L(0, w_0 \nu).\]
	Since $T$ sends the spherical vector to the spherical vector, the image of $T$ is contained in $L^0(0,w_0 \nu) = V(0,w_0 \nu)$.
	As a consequence, $T$ factors
	\[L(0,\nu) \to V(0,\nu) \to V(0,w_0 \nu) \subset L(0,w_0 \nu).\]
	Now Schur's lemma implies the claim.

	Take a nonzero $T \in \Hom_{D^\Sigma}(L(0,\nu^\Sigma),L(0,w_0 \nu^\Sigma))$. The claim shows such $T$ is unique up to a scalar factor. In particular, the image of $T$ is $V(0,w_0 \nu^\Sigma)$.
	Now using the isomorphism $L(0,\nu) \simeq {\rm Ind}_\Sigma^\Pi (L(0,\nu^\Sigma),\nu^{\perp \Sigma}) = D_c \otimes_{B_\Sigma} L(0,\nu^\Sigma)_{(0,(\nu-2\rho)^{\perp \Sigma})}$,
	\[1 \otimes_{B_\Sigma} T \in \Hom_{D^\Sigma}(L(0,\nu),L(0,w_0 \nu)).\]
	Hence
	\[V(0,\nu) \simeq {\rm Im} (1 \otimes_{B_\Sigma} T) = {\rm Ind}_\Sigma^\Pi (V(0,w_0 \nu^\Sigma),\nu^{\perp \Sigma}) \simeq {\rm Ind}^\Pi_\Sigma(V(0,\nu^\Sigma),\nu^{\perp \Sigma}).\]
	(We again used the claim for the first isomorphism.)
	\end{proof}
	\section{Intertwining operators}
	In this section, we complete the classification of unitarizable admissible spherical irreducible $D$-modules for $\fg = \frsl_3$.

	Let us begin with an easy obstruction for the unitarizability coming from the norm estimate which contains the case in Corollary \ref{thm integral}.

	\begin{lem}\label{thm fd unitary}
	Let $\mu \in \fh^*$ such that $V(0,\mu)$ is unitarizable. Then for any $\lambda \in P_+$,
	\[|{\rm Tr}_\lambda(K_{\mu})| \leq {\rm Tr}_\lambda(K_{2 \rho}).\]
	In particular, if $\Re\mu - 2 \rho$ is nonzero and dominant, then $V(0,\mu+2\rho)$ is not unitarizable.
	\end{lem}
	\begin{proof}
	Recall that for $v \in V(0,\mu)^0$,
	\[\chi_q(\lambda) v = K_{\mu-2\rho}(\chi_q(\lambda)) v = {\rm Tr}_\lambda(K_\mu) v.\]
	Suppose $V(0,\mu)$ is unitarizable. Since $\| \chi_q(\lambda) \| \leq {\rm Tr}_\lambda(K_{2\rho})$,
	\[|K_{\mu-2\rho}(\chi_q(\lambda))| \leq {\rm Tr}_\lambda (K_{2\rho}).\]
	\end{proof}

	In the rank $1$ case, we have the following expression for the intertwining operator $T \in \Hom_D(V(0,\nu),V(0,-\nu))$.
	\begin{prop}
	Let $\fg = \frsl_2$. Fix $\nu \in {\mathbb C}/2\pi i \log(q)^{-1} {\mathbb Z}$.
	Suppose $\nu \not \in ({\mathbb Z} \setminus \{0\} ) + \pi i \log(q)^{-1}{\mathbb Z}$.
	Put
	\[T^s = - \prod_{0 \leq r \leq s} \frac{q^{r-\nu} - q^{-r+\nu}}{q^{r+\nu} -q^{-r-\nu}}\]
	and define $T \in \End(L)$ by
	\[T\omega := T^s \omega\]
	for $\omega \in L^s$.
	Then with respect to the identification $L \simeq V(0,\nu) \simeq V(0,-\nu)$,
	\[T \in \Hom_D(V(0,\nu),V(0,-\nu)).\]
	\end{prop}
	\begin{proof}
	From Proposition \ref{thm HC}, there exists an intertwining operator $T \in \Hom_D(V(0,\nu),V(0,-\nu))$.
	For $s \in {\mathbb Z}_{\geq 0}$, since $[V(0,\nu):V(s)] = 1$, $T$ is a scalar on $V(0,\nu)^s$. Let $T^s := T|_{V(0,\nu)^s}$.
	Multiplying a scalar if necessary, we may assume $T_0 = 1$.

	The Clebsch-Gordan rule asserts $c^r a^r$ is in $\bigoplus_{0 \leq k \leq r} \cO(SU_q(2))^k$, hence $\varphi \triangleleft (c^r a^r) \in \bigoplus_{0 \leq k \leq r} L^k$. Moreover since $\varphi \triangleleft (c^r a^r) \in L$ is of weight $r$, $\varphi \triangleleft (c^r a^r) \in L^r$.
	Moreover since $\varphi$ is faithful on $\cO(SU_q(2))$, $\varphi \triangleleft(c^r a^r)$ is nonzero.
	We have
	\begin{align*}
	\pi^\nu(c) \Lambda(\varphi \triangleleft (c^r a^r)) & = q^{-1+\nu} \Lambda(a \triangleright \varphi \triangleleft (c^r a^r S(c))) + q^{1-\nu} \Lambda(c \triangleright \varphi \triangleleft (c^r a^r S(d))) \\
	& = - q^{\nu} \Lambda(a \triangleright \varphi \triangleleft (c^r a^r c)) + q^{1-\nu} \Lambda(c \triangleright \varphi \triangleleft(c^r a^{r+1})) \\
	& = q^{r+1}(- q^{r+1+\nu} + q^{-r-1-\nu}) \Lambda(\varphi \triangleleft (c^{r+1} a^{r+1})).
	\end{align*}
	Hence
	\begin{align*}
	&T_{r+1} q^{r+1}(- q^{r+1+\nu} + q^{-r-1-\nu}) \Lambda(\varphi \triangleleft (c^{r+1} a^{r+1})) \\
	& = T \pi^\nu(c) \Lambda(\varphi \triangleleft (c^r a^r)) \\
	& = \pi^{-\nu}(c) T \Lambda(\varphi \triangleleft (c^r a^r)) \\
	& = T_r q^{r+1}(-q^{r+1-\nu} + q^{-r-1+\nu}) \Lambda(\varphi \triangleleft (c^r a^r)).
	\end{align*}
	Therefore
	\[\frac{T_{r+1}}{T_r} = \frac{q^{r+1-\nu} - q^{-r-1+\nu}}{q^{r+1+\nu}-q^{-r-1-\nu}}.\]
	Iterating use of this formula shows the conclusion.
	\end{proof}

	One can classify all irreducible unitary spherical representations of the Drinfeld double of $SU_q(2)$ as follows, which has already appeared in \cite{Pus}.
	\begin{cor}
	Let $\fg = \frsl_2$. Then for $\nu \in {\mathbb C}/2 \pi i \log(q)^{-1} \mathbb Z$, $V(0,\nu)$ is unitarizable if and only if $\nu \in i {\mathbb R}$ or $\nu \in {\mathbb R} + \pi i \log(q)^{-1} {\mathbb Z}$ such that $|\Re\nu| \leq 1$.
	\end{cor}
	\begin{proof}
	In the case of $\nu \in ({\mathbb Z} \setminus \{0\}) + \pi i \log(q)^{-1} {\mathbb Z}$, Lemma \ref{thm fd unitary} asserts that if $V(0,\nu)$ is unitarizable, then $|\Re\nu| = 1$. Conversely if $|\Re \nu| = 1$, $V(0,\nu)$ is $1$-dimensional and unitary.

	Let $\nu \not \in ({\mathbb Z} \setminus \{0\}) + \pi i \log(q)^{-1} {\mathbb Z}$.
	Then $V(0,\nu) = L(0,\nu)$.
	From Corollary \ref{thm sesqui-form}, $V(0,\nu)$ is not unitarizable unless there exists $w \in W$ such that $-\overline{\nu} = w\nu$.
	Suppose $\nu$ satisfies the condition.
	Take $T \in \Hom_D(V(0,\nu),V(0,w\nu))$ such that $T^0 = 1$.
	Then the invariant sesquilinear form is given by
	\[(x,y)_0 := (Tx,y),\]
	where the inner product on the left hand side is the canonical one on $L$ given by
	\[(x,y) = \dphi(y^* x).\]
	Since $(\varphi,\varphi)_0 = 1$, $V(0,\nu)$ is unitarizable if and only if $T$ is positive definite on $L$.

	In the case of $w = 1$ (that is, $\nu \in i{\mathbb R}$), $T = 1$. Hence this is a priori positive definite.

	In the case of $w = -1$ (that is, $\nu \in {\mathbb R} + \pi i \log(q)^{-1} {\mathbb Z}$), $T$ is positive definite if and only if
	\[-\prod_{0 \leq r \leq n}\frac{q^{r+\nu} - q^{-r-\nu}}{q^{r-\nu} - q^{-r+\nu}} \geq 0\]
	for any $n$. This holds if and only if $|\Re \nu| < 1$.
	\end{proof}
	Finally, we get a classification of the unitary spherical irreducible representations of $D$ for $\fg = \frsl_3$.
	\begin{thm}\label{thm sl3}
	Let $\fg = \frsl_3$ and take $\nu \in X$.
	Then $V(0,\nu)$ is unitarizable if and only if
	\begin{enumerate}[{\rm (i)}]
	\item $\nu \in i \fh_{\mathbb R}$,
	\item $\nu \in W(t \alpha + i \mu)$ for $t \in [-1,1]$, $(\mu,\alpha) \in 2\pi \log(q)^{-1} {\mathbb Z}$ or
	\item $\nu \in W(2\rho + 2\pi i \log(q)^{-1} P)$.
	\end{enumerate}
	\end{thm}
	\begin{proof}
	Put $\Pi = \{\alpha,\beta\}$.

	Take $\nu \in X$. If there is no $w \in W$ such that $-\overline{\nu} = w \nu$, $V(0,\nu)$ is not unitarizable.

	If $-\overline{\nu}= w \nu$ for some $w \in W$, we can divide the case into the following.

	\begin{enumerate}[(1)]
	\item Suppose $w = e$, that is, $\nu \in i \fh_{\mathbb R}$. In this case, the intertwining operator is just identity, hence positive definite, so $V(0,\nu)$ is unitarizable.

	\item Suppose $w = s_\alpha$, that is, $\nu = t \alpha + i \mu$ for $t \in {\mathbb R}$ and $(\mu,\alpha) \in 2 \pi \log(q)^{-1} {\mathbb Z}$.

	First suppose $\mu \not \in 2 \pi \log(q)^{-1} P$ or $t \not \in 2{\mathbb Z} \setminus \{0\}$. Then
	from Proposition \ref{thm irrep ind},
	\[V(0,\nu) = {\rm Ind}_\alpha^\Pi (V(0,t+i(\mu,\alpha)\alpha/2),i(\mu-(\mu,\alpha)\alpha/2)).\]
	In this case, Corollary \ref{thm ind form} implies $V(0,\nu)$ is unitarizable if and only if $V(0,t+i(\mu,\alpha)/2)$ is. This holds if and only if $|t| \leq 1$.

	Suppose $\mu \in 2 \pi \log(q)^{-1} P$ and $t \in 2 {\mathbb Z} \setminus \{0\}$.
	In this case, $V(0,\nu)$ is finite dimensional.
	Thanks to Lemma \ref{thm fd unitary}, $V(0,\nu)$ is not unitarizable unless $\nu \in W(2 \alpha + 2\pi i\log(q)^{-1} P) = W (2\rho + 2\pi i \log(q)^{-1} P)$.
	On the other hand, if $\nu \in W(2\rho+2\pi i \log(q)^{-1}P)$, $V(0,\nu)$ is $1$-dimensional and unitary.
	\item For case of $w = s_\beta$ or $s_\alpha s_\beta s_\alpha$, there exists $w' \in W$ such that $w' \nu$ is in the case (2). Since $V(0,\nu)$ is isomorphic to $V(0,w' \nu)$, we already have done.

	\item In the case of $w = s_\alpha s_\beta$ or $s_\beta s_\alpha$, there is no $\nu \in X$ satisfying $-\overline{\nu} = w \nu$ except for $\nu \in 2 \pi i \log(q)^{-1} P$, which is already considered in the case (1).
	\end{enumerate}
	Combining these cases, we get the desired conclusion.
	\end{proof}
	We conclude the content of this paper with another application of the whole theory.
	\begin{thm}
	The Drinfeld double of $SU_q(2n+1)$ has property {\rm (T)} for any $n \in {\mathbb Z}_+$. 
	\end{thm}
	\begin{proof}
	First notice that the restriction of the topology on $X/W$ coincides with the Fell topology on the subset of the unitary dual formed by the spherical irreducible representations. Hence we only need to show $V(0,\nu)$ is not unitarizable for any $\nu \in X$ sufficiently close to $2\rho$.
	
	Let $\fg = \frsl_{2n+1}$. We identify $\fh^*$ with
	\[\{\nu = (\nu_n,\nu_{n-1},\dots,\nu_{-n}) \in {\mathbb C}^{2n+1} \mid \sum_{k=-n}^n \nu_k = 0\}.\]
	Let $(e_k)_{k=n,n-1,\dots,-n}$ be the canonical basis of ${\mathbb C}^{2n+1}$ and fix the set of simple roots by $\alpha_k := e_k - e_{k-1}$ ($k=n,n-1,\dots,-n+1$). 
	Then the element $2\rho$ corresponds to the sequence $(2n,2n-2,\dots,-2n)$.

	For $\nu \in \fh^*$, consider the following two conditions.
	\begin{enumerate}[(1)]
	\item $|\nu_k - 2k|  < {\rm min}\{1, \pi \log(q)^{-1}\}$ for any $k$.
	\item For any proper subset $p \subset \{n,n-1,\dots,-n\}$ which contains $n$ and is closed under $k \mapsto -k$, we have
	\[\sum_{k \in p} |q^{\nu_k} | > \sum_{l = 1}^{\#p} q^{-\#p + 2l - 1}.\]
	\end{enumerate}
	It is easy to observe that $2\rho$ satisfies these conditions. Hence if $\nu \in \fh^*$ is sufficiently close to $2\rho$, then $\nu$ also satisfies the conditions.
	We show $V(0,\nu)$ is not unitarizable for such $2\rho \neq \nu \in \fh^*$.
	
	Suppose $V(0,\nu)$ is unitarizable.
	Since $-\overline{\nu}$ is a permutation of $\nu$ modulo $2 \pi i \log(q)^{-1} Q$, the condition (1) implies
	\[\nu_{-k} = -\overline{\nu_k}.\]
	Let $p \subset \{n,n-1,\dots,-n\}$ be the set of all indices $k$ such that
	\[\nu_k - \nu_n \text{ or } -\overline{\nu_k} - \nu_n \in 2 {\mathbb Z} + 2 \pi i \log(q)^{-1} {\mathbb Z}.\]
	By definition, $k \in p$ if and only if $-k \in p$.

	We claim $p \neq \{-n,-n+1,\dots,n\}$. (For this, we use the assumption that the rank is odd.)
	Suppose $p = \{-n,-n+1,\dots,n\}$.
	Since $\nu_0$ is imaginary, the definition of $p$ and the condition (1) imply
	\[\nu_k - \nu_0 \in 2 {\mathbb Z}\]
	for any $k$.
	Again from the condition (1), $\nu$ is of the form
	\[\nu = (2n+\nu_0, 2n-2+\nu_0, \dots, -2n + \nu_0).\]
	Combining with $\displaystyle \sum_k \nu_k = 0$, we get $\nu_0 = 0$, which contradicts with $\nu \neq 2\rho$.

	Now take a permutation $\tilde \nu$ of $\nu$ with the following property:
	\begin{itemize}
	\item $\tilde \nu_n = \nu_n$,
	\item the set of indices corresponding to $p$ is $\{n,n-1,\dots,l\}$ and
	\item the sequence $\Re\tilde \nu$ is decreasing on $\{n,n-1,\dots,l\}$ and $\{l-1,l-2,\dots,-n\}$.
	\end{itemize}
	Put $\Sigma_1 := \{\alpha_k \mid l+1 \leq k \leq n\}$, $\Sigma_2 = \{\alpha_k \mid -n+1 \leq k \leq l-1\}$ and $\Sigma := \Sigma_1 \amalg \Sigma_2$.
	Since $\Sigma$ satisfies the assumption in Proposition \ref{thm irrep ind},
	\[V(0,\tilde{\nu}) = {\rm Ind}^\Pi_\Sigma (V(0,\tilde{\nu}^\Sigma),\tilde{\nu}^{\perp \Sigma}).\]
	Put $\tilde \nu^i := \tilde{\nu}^{\Sigma_i}$ for $i = 1,2$. Then we have
	\[V(0,\tilde{\nu}^\Sigma) = V(0,\tilde \nu^1) \otimes V(0,\tilde \nu^2)\]
	as $D^\Sigma = D^{\Sigma^1} \otimes D^{\Sigma^2}$-module.

	On the other hand, we compute
	\begin{align*}
	\tilde \nu^1 & = (\tilde \nu_n - t, \tilde \nu_{n-1} - t, \dots, \tilde \nu_l - t, 0, \dots, 0), \\
	\tilde \nu^2 & = (0, \dots, 0, \tilde \nu_{l-1} - s, \tilde \nu_{l-2} - s, \dots, \tilde \nu_{-n} -s) \\
	\tilde \nu^{\perp \Sigma} & = (t,t, \dots, t, s, s, \dots s)
	\end{align*}
	for $\displaystyle t := \frac{1}{n-l+1} \sum_{k=l}^n \tilde \nu_k$ and $\displaystyle s := \frac{1}{l+n} \sum_{k=-n}^{l-1} \tilde \nu_k$.
	Since $\nu_{-k} = -\overline{\nu_k}$ and $p$ is symmetric around zero, the numbers $t$ and $s$ are imaginary.
	Therefore Corollary \ref{thm ind form} implies $V(0,\tilde{\nu}^\Sigma)$ is unitarizable.
	Hence the modules $V(0,\tilde{\nu}^i)$ are also unitarizable.

	To conclude the proof, we will show $V(0,\tilde{\nu}^1)$ is not unitarizable using Lemma \ref{thm fd unitary}. Let $\lambda$ be the fundamental weight of $\fg^\Sigma$.
	Then the inequality
	\[{\rm Tr}_\lambda(K_{\tilde{\nu}^1}) > {\rm Tr}_\lambda(K_{2\rho^{\Sigma_1}})\]
	follows from the condition (2).
	Thanks to Lemma \ref{thm fd unitary}, $V(0,\tilde{\nu^1})$ is not unitarizable.
	\end{proof}
\section{Appendix.\ Central property (T) for discrete quantum groups}
	In \cite{DFY}, a central Haagerup property and a central completely contractive approximation property for discrete quantum groups are introduced. In this appendix we define in a similar way a central property (T).
	For basic definitions, we refer \cite{DFY}.

	Let $G$ be a compact quantum group. Let $C^u(G)$ be the universal function algebra on $G$ and $C(G)$ the reduced function algebra.
	We denote the character algebra in $C^u(G)$ by $C^u({\rm Char}(G))$.
	\begin{df}\label{def (T)}
	We say $\dG$ has {\it central property ${\rm (T)}$} if the following hold:

	If a net of central states $(\omega_i) \subset C^u(G)^*$ converges to $\varepsilon$ in the weak*-topology, then it converges in norm.
	\end{df}
	\begin{rem}
	According to \cite[Theorem 3.1]{Kyed}, property (T) in the sense of \cite{Fima} is equivalent to the following:

	If a net of states $(\omega_i)$ converges to $\varepsilon$ in the weak*-topology, then it converges in norm.

	In particular, property (T) implies central property (T).
	\end{rem}
	We start with the following standard observation. For a standard properties of the Fell topology on the unitary dual of a C*-algebra, we refer \cite{Dix}.
	\begin{lem}\label{thm (T)C*}
	Let $A$ be a C*-algebra and $\chi$ a $*$-character of $A$. Then the following are equivalent.
	\begin{enumerate}[{\rm (i)}]
	\item The character $\chi$ is isolated in the unitary dual of $A$.
	\item If a net of states $(\omega_i)$ converges to $\chi$ in the weak*-topology, then it converges in norm.
	\item There exists a central projection $p^\chi$ in $A$ such that for any $a \in A$,
	\[a p^\chi = \chi(a) p^\chi.\]
	\end{enumerate}
	\end{lem}
	\begin{proof}
	The equivalence between {\rm (i)} and {\rm (ii)} is immediate from the definition of the topology on the unitary dual. From {\rm (ii)} to {\rm (iii)}, see \cite[Theorem 17.2.4]{BO}. (The proof works for general C*-algebras.)

	From {\rm (iii)} to {\rm (i)}, applying \cite[Proposition 3.2.1]{Dix} to $A = p^\chi A \oplus (1-p^\chi) A$, we can decompose the unitary dual of $A$ into a disjoint union of those of $p^\chi A$ and $(1-p^\chi) A$.
	Since the unitary dual of $p^\chi A$ is $\{\chi\}$, $\chi$ is isolated.
	\end{proof}
	For $\omega \in C^u(G)^*$, one can define a multiplier $T^u_\omega$ on $C^u(G)$
	\[T_\omega^u := (\omega \otimes {\rm id}) \Delta^u.\]
	Then since $\omega = \varepsilon \circ T_\omega^u$, we get
	\[\| \omega \| = \| T_\omega^u \|.\]

	If we start with a central multiplier $T^u$ on $C^u(G)$, that is, a completely bounded map from $C^u(G)$ to itself which is equivariant under the left-right action of $G$,
	then $\omega_T := \varepsilon \circ T^u$ is a central state. Hence we get a one-to-one correspondence between central completely bounded multipliers and central bounded functionals which preserves the norm.

	In particular, we get the following.
	\begin{prop}
	The following are equivalent.
	\begin{enumerate}[{\rm (i)}]
	\item The quantum group $\dG$ has central property ${\rm (T)}$.
	\item If a net of central completely positive multipliers $(T_i^u)$ on $C^u(G)$ converges to the identity pointwisely, then it converges in norm.
	\end{enumerate}
	\end{prop}
	In the completely same way as in \cite[Proposition 6.3]{Fre}, we get the following.
	\begin{prop}\label{thm mon-equiv}
	Let $G$ and $H$ compact quantum groups which are monoidally equivalent.
	Then we have a one-to-one correspondence between central completely bounded multipliers on $C^u(G)$ and $C^u(H)$ preserving the norm and the weak*-topology.

	In particular, if $\dG$ has central property ${\rm (T)}$, so is $\widehat{H}$.
	\end{prop}
	In the Kac type case, central property (T) is equivalent to original property (T).
	To show this, we prepare a lemma.
	\begin{lem}
	Let $G$ be a compact quantum group of Kac type. Then there exists a conditional expectation
	\[E:C^u(G) \to C^u({\rm Char}(G))\]
	such that
	\[E(u^\pi_{ij}) = \frac{1}{\dim \pi} \chi(\pi) \delta_{ij}.\]
	\end{lem}
	\begin{proof}
	Let us consider the left and right regular representations:
	\[\dlambda, \drho: C(G) \to B(L^2(G)).\]
	Notice that
	\[(\dlambda(x_{(1)}) \drho(x_{(2)}) \Omega, \Omega) = \varepsilon(x),\]
	where $\Omega$ is the canonical cyclic vector in $L^2(G)$.

	First, we claim
	\[\alpha: C^u(G) \to (C(G) \otimes C(G)) \otimes_{{\rm max}} C(G): x \mapsto (x_{(3)} \otimes x_{(1)}) \otimes x_{(2)}\]
	is injective. Using Fell's absorption,
	\[\pi: (C(G) \otimes C(G))\otimes_{{\rm max}} C(G) \to B(L^2(G) \otimes L^2(G)) \otimes C^u(G) : x \otimes y \otimes z \mapsto \dlambda(y) \drho(z_{(1)}) \otimes \dlambda(x) \drho(z_{(3)}) \otimes z_{(2)}\]
	is a well-defined representation.
	Now
	\[{\rm id} = (\omega_\Omega \otimes \omega_\Omega \otimes {\rm id}) \circ \pi \circ \alpha\]
	shows $\alpha$ is injective.
	
	Finally consider the $\varphi \otimes \varphi$-preserving conditional expectation
	\[E^\Delta: C(G) \otimes C(G) \to \Delta(C(G)).\]
	Then the image of $(E^\Delta \otimes {\rm id})\alpha$ is $\alpha(C^u(\mathrm{Char}G))$ and
	\[E = \alpha^{-1} (E^\Delta \otimes {\rm id})\alpha\]
	is the desired conditional expectation.
	\end{proof}
	\begin{prop}\label{Thm (T)-char}
	Let $G$ be a compact quantum group of Kac type. Then the following are equivalent.
	\begin{enumerate}[{\rm (i)}]
	\item The quantum group $\dG$ has property $\rm{(T)}$.
	\item The quantum group $\dG$ has central property $\rm{(T)}$.
	\item If a net of states $(\omega_i)$ on the character algebra $C^u({\rm Char}(G))$ converges to $\varepsilon$ in the weak*-topology, it converges in norm.
	\end{enumerate}
	\end{prop}
	\begin{proof}
	As we observed, {\rm (i)} implies {\rm (ii)}.

	To show {\rm (ii)} implies {\rm (iii)}, suppose $\dG$ has property (T). Take a net of states $(\omega_i)$ on $C^u({\rm Char(G)})$ converging to $\varepsilon$ in the weak*-topology.
	By the previous lemma, there exists a conditional expectation
	\[E:C^u(G) \to C^u({\rm Char}(G)).\]
	Then $(\omega_i \circ E)$ is a net of central states on $C^u(G)$ converging to $\varepsilon$ in the weak*-topology. Hence it converges in norm. Since $\omega_i \circ E$ is nothing but $\omega_i$ on $C^u({\rm Char}(G))$, $(\omega_i)$ converges to $\varepsilon$ in norm.

	We show {\rm (iii)} implies {\rm (i)}.
	Take the projection $p^\varepsilon \in C^u({\rm Char}(G))$ as in Lemma \ref{thm (T)C*}. We claim $x p^\varepsilon = p^\varepsilon x = \varepsilon(x) p^\varepsilon$. To see this, embed $C^u(G)$ into $B(H)$ for a Hilbert space $H$. Then
	\[\chi(\pi) p^\varepsilon \xi = \varepsilon(\chi(\pi)) p^\varepsilon \xi = \dim(\pi) p^\varepsilon \xi.\]
	Since $u^\pi_{ii}$ is a contraction, this shows
	\[u^\pi_{ii} p^\varepsilon \xi = p^\varepsilon \xi.\]
	In a similar way,
	\[(u^\pi_{ii})^* p^\varepsilon \xi = p^\varepsilon \xi.\]
	Since $\sum_k (u^\pi_{ki})^* u^\pi_{kj} = 1$,
	\[u^\pi_{ij} p^\varepsilon \xi = \delta_{ij} p^\varepsilon \xi,\]
	that is,
	\[x p^\varepsilon = \varepsilon(x) p^\varepsilon.\]
	Taking $*$ of the formula, we get $p^\varepsilon$ is central. Again from Lemma \ref{thm (T)C*}, $\dG$ has property (T).
	\end{proof}
	\begin{prop}\label{(T) of double}
	Suppose the Drinfeld double $D(G)$ has property ${\rm (T)}$. Then $\dG$ has central property {\rm (T)}.
	\end{prop}
	\begin{proof}
	Thanks to Lemma \ref{thm (T)C*}, the assumption is equivalent to the following:
	If a net of states on $\cO(G) \bowtie c_c(\dG)$ converges to the counit $\varepsilon_D$, then it converges in norm.

	Take central states $\omega,\mu$ on $C^u(G)$,
	Let $\pi:C^u(G) \to M(C^u_0(\widehat{D}(G)))$ be the $*$-homomorphism extending $\cO(G) \to M(\cO(G) \bowtie c_c(\dG))$. Then since $\omega = {\rm Ind}\omega \circ \pi$,
	\[\| \omega - \mu \| = \| ({\rm Ind}\omega - {\rm Ind}\mu) \circ \pi \| \leq \| {\rm Ind}\omega - {\rm Ind}\mu \|.\]
	In particular, for a net of central states $(\omega_i)$, if $({\rm Ind}\omega_i)$ converges to ${\rm Ind} \omega$ in norm, $(\omega_i)$ converges to $\omega$ in norm.

	Hence if we have a net of central states $(\omega_i)$ converges to $\varepsilon$ in the weak*-topology, then $({\rm Ind}\omega_i)$ converges to $\varepsilon_D = {\rm Ind}\varepsilon$ in the weak*-topology.
	By assumption, it converges in norm, which turns out to be the convergence of $(\omega_i)$ to $\varepsilon$ in norm, which is desired.
	\end{proof}
	\begin{cor}
	The discrete quantum group $\widehat{SU_q(2n+1)}$ has central property ${\rm (T)}$.
	\end{cor}
	As another application of Theorem \ref{thm sl3}, we get the following restriction on unitary fiber functors of ${\rm Rep}(SU_q(3))$.
	\begin{prop}
	Let $F: {\rm Rep}(SU_q(3)) \to {\rm Hilb}_f$ be a unitary fiber functor. Let $x \in {\rm Rep}(SU_q(3))$ be the fundamental representation of $SU_q(3)$. Then we have either $\dim F(x) = q^2 + 1 + q^{-2}$ or $\dim F(x) \leq q + 1 + q^{-1}$.
	\end{prop}
	\begin{proof}
	In \cite{DFY}, it is shown that for any compact quantum group $G$,
	\[\omega: \cO(G) \to {\mathbb C} : u^\pi_{ij} \mapsto \frac{\dim(\pi)}{\dim_q(\pi)} \delta_{ij}\]
	is a central state.
		
	For a unitary fiber functor $F$, take the corresponding quantum group $G$ which is monoidally equivalent to $SU_q(3)$.
	Then by Proposition \ref{thm mon-equiv}, 
	\[\omega_F: u^\pi_{ij} \mapsto \frac{\dim F(\pi)}{\dim_q(\pi)} \delta_{ij}\]
	is a central state on $SU_q(3)$.
	On the other hand, the central states corresponding to $\nu \in X$ is nothing but
	\[\omega_\nu: u^\pi_{ij} \mapsto \frac{{\rm Tr}(K_\nu)}{\dim_q(\pi)} \delta_{ij}.\]
	Put $N := \dim F(x)$ and find $t \geq 0$ such that 
	\[N = q^t + 1 + q^{-t}.\]
	Then $\omega_{t \rho} = \omega_F$. Therefore it is positive definite if and only if $t \rho$ is in the list in Theorem \ref{thm sl3}, hence $t \leq 1$ or $t = 2$.
	\end{proof}

\end{document}